\documentclass[a4paper,11pt]{article}
\usepackage[T1]{fontenc}
\usepackage[utf8]{inputenc}
\usepackage{lmodern}
\usepackage[russian,english]{babel}
\usepackage{amsmath, amsfonts, amssymb, graphicx,amsthm}
\usepackage{amsfonts,amssymb}
\usepackage[backgroundcolor=yellow!65!white]{todonotes}
\usepackage{nameref}
\usepackage[hypcap=true]{caption}
\usepackage[hypcap=true]{subcaption}
\usepackage{csquotes}

\graphicspath{{pictures/}}

\newtheorem{prop}{Proposition}
\newtheorem{thm}{Theorem}
\newtheorem{lem}{Lemma}

\newtheorem{conj}{Conjecture}
\theoremstyle{definition}
\newtheorem{defin}{Definition}

\theoremstyle{remark}
\newtheorem{rem}{Remark}

\usepackage{hyperref}

\DeclareMathOperator{\Sing}{Sing}
\DeclareMathOperator{\Sep}{Sep}
\DeclareMathOperator{\Per}{Per}

\def\t {\tilde }

\def\g {\gamma }

\def\G {\Gamma }

\def\e {\varepsilon }
\def\eps {\varepsilon }

\def\zz{\mathbb{Z}}
\def\bbZ{\mathbb{Z}}

\def\rr{\mathbb R}
\def\bbR{\mathbb R}

\def\nn{\mathbb N}

\def\C{\mathcal C}

 \def\ph {\varphi }

 \def\om{\omega }

  \def\nbd{neighborhood }

    \def\nbds{neighborhoods }

\def\lc{limit cycle }

    \def\be{\begin{equation}}

     \def\ee{\end{equation}}

    \def\st{such that }

   \def\tes{there exists }

       \def\g{\gamma }

\def\t{\tilde }

\def\lbs{large bifurcation support }

\def\t{\tilde }

\def\vf{vector field }

\def\d{\delta }

\def\a{\alpha}

\def\nbds{neighborhoods }

\def\C{\mathcal C}

\def\P{\mathcal P}

\def\diffeo {diffeomorpism }

\def\lb {local bifurcations }
\def\nb {nonlocal bifurcations }
\def\vfs {vector fields }
\def\vf {vector field }

\def\bif {bifurcation }

\def\tes {there exists }

\def\pc {polycycle }

\def\pmap {Poincaré map }

\def\pc {\P \C }
\def \pmap {Poincaré map }
\def \lbs {large bifurcation support }

\def \nccs {transversal loops }
\def\ns {non-synchronized }

\def\La {\Lambda }

\title{Global bifurcations in generic one-parameter families \\with a parabolic cycle on $S^2$}
\author{N. Goncharuk, \thanks{ Cornell University }\
\and{Yu. Ilyashenko\thanks
    {National Research University Higher School of Economics, Russia}\  \thanks{ Independent University of Moscow}\ \thanks{ The authors were supported in part by the grant RFBR 16-01-00748}} \and {N. Solodovnikov\thanks
    {National Research University Higher School of Economics, Russia}\ }}

\date{}

\begin{document}

\maketitle

\begin{abstract} We classify global bifurcations in generic one-parameter local families of \vfs on $S^2$ with a parabolic cycle. The classification is quite different from the classical results presented in  monographs on the bifurcation theory. As a by product we prove that generic families described  above are structurally stable.
\end{abstract}



\section{Introduction}   \label{sec:intro}

\subsection{Main results}

This article is a part of a larger investigation whose main goals are:

\begin{itemize}
  \item To prove structural stability of generic one-parameter families of \vfs in the two-sphere;
  
  \item To give a complete classification of the bifurcations in these families with respect to the weak equivalence relation (the definition is recalled below).
\end{itemize}

These goals are achieved in \cite{IS}, \cite{St}, and the present paper.

Structural stability result is well expected. It was predicted in \cite{S}; the sketch of the proof (of another but close result) was given in \cite{MP} with the note: ``Full proofs will appear in a forthcoming paper.'' To the best of our knowledge, that paper was not written. Here we prove structural stability for unfoldings of parabolic limit cycles, which constitutes the first of the two main results of our paper.

Complete classification of the bifurcations seems to be quite unexpected. Global bifurcations in generic families that unfold a vector field with a separatrix loop are characterized by a finite set on a circle considered up to a homeomorphism \cite{IS}. 
Global bifurcations in generic families that unfold a vector field with a parabolic cycle are characterized by two finite sets on a coordinate circle $\rr/ \zz$,  considered up to a certain equivalence relation (see Definition \ref{def:seteq}).
The latter result is the second main theorem of our paper. The precise statements follow (see Theorem \ref{thm:stab} and Theorem \ref{thm:class}).

Together with \cite{IS} and \cite{St}, this paper achieves the goals stated at the beginning, and thus completes the study of global bifurcations in one-parameter families.

These results are a part of a large program of the development of the global bifurcation theory on the sphere outlined in \cite{I16}.  There was a belief, formulated by V.Arnold as a conjecture in \cite[Sec. I.3.2.8]{AAIS}, that a generic family is structurally stable up to a weak equivalence: close finite-parameter families are weakly equivalent.
This natural conjecture turns out to be false for 3-parameter families, as was proved in a recent work \cite{IKS} (with weak equivalence replaced by moderate equivalence, which is a technical difference). Namely, the authors prove that the moderate classification of families with “tears of the heart” polycycle has numerical moduli, and generic family of this class is \emph{not} structurally stable. The effect is due to \emph{sparkling saddle connections} that accumulate to the polycycle; their order is different for close families, which implies the statement.

Now the following  problems arise:

\begin{itemize}
\item Find out whether 2-parameter families of vector fields on $S^2$ are structurally stable  (up to the weak equivalence);

\item Classify their bifurcations;

\item Distinguish structurally stable three-parameter families from the unstable ones, and find new examples of structurally unstable three-parameter families.
\end{itemize}

These problems are natural steps that follow the present paper. Let us pass to the detailed presentation.

\vskip 1 cm

By default, \emph{\vfs } below are infinitely smooth \vfs on $S^2$ with isolated singular points, and \emph{families} are families of \vfs on $S^2$. The sphere is oriented, and all the homeomorphisms of $S^2$ under consideration preserve orientation.

\subsection{Vector fields of class $PC$}   \label{subsec:pc}

\begin{defin}   \emph{A  $C^2$ \vf of class $PC$} is a \vf with a parabolic limit cycle $\g $ and no other degeneracies. Namely,
the following assumptions hold:
\begin{itemize}
\item all the singular points and limit cycles of the \vf except for $\g $ are hyperbolic;

\item the \vf has no saddle connections;

\item the parabolic cycle $\g $ is of multiplicity 2, that is, its \pmap has the form $x \mapsto x + ax^2 + \dots , \ a \ne 0$.
\end{itemize}
\end{defin}

Vector fields of class $PC$ form an immersed Banach manifold of codimension one in the space of $C^r$-smooth \vfs on $S^2$, $r\ge 3$, see \cite[Proposition 2.2]{S}.
\subsection{Structural stability}  \label{subsec:stst}

Let us recall some basic definitions.

\begin{defin}
	Two vector fields $v$ and $w$ on $S^2$ are called \emph{orbitally topologically equivalent}, if there exists a homeomorphism $S^2\to S^2$ that 
	links the phase portraits
	of $v$ and $w$, that is, sends orbits of $v$ to orbits of $w$ and preserves their time orientation.
\end{defin}

In this article, we consider one-parameter families of vector fields on $B\times S^2$. Here $B\subset \bbR$ is the base of the family. We work with local families with bases $(\bbR, 0)$, in the following sense. 
\begin{defin}
     A \emph{local family} of vector fields at $\e = 0$ is a germ at $\{0\} \times S^2$ of a family on $B \times S^2$, where $B \owns 0$, $B \subset \rr$ is open.
\end{defin}
\begin{defin}
    \label{def:unfol}
    An \emph{unfolding} of a vector field $v$ is a local family for which $v$ corresponds to zero parameter value.
    We say that this family \emph{unfolds} the vector field.
\end{defin}

The following definition lists the notions of equivalence for local families of vector fields that  we will use in this paper.
\begin{defin}
    \label{def:weak-equiv-1}
	Let $B$, $B'$ be topological balls in $\mathbb{R}$ that contain $0$.
	Two local families of vector fields $\{v_{\alpha}, \alpha \in B \}$,
	$\{w_{\beta},\beta\in B' \}$ on $S^2$ are called 
	\begin{itemize}
	 \item 	\emph{weakly topologically equivalent} if there exists a map
	\begin{align*}
		H&\colon
		B\times S^2 \to B\times S^2,&H(\alpha,x) &= (h(\alpha),H_{\alpha}(x))
	\end{align*}
	such that $h$ is a homeomorphism of the bases, $h(0)=0$, and for each $\alpha\in B$ the map $H_{\alpha}:S^2\to S^2$ is a
	homeomorphism that takes the phase portrait of $v_{\alpha}$ to the phase portrait of $w_{h(\alpha)}$. 
	
	\item $sing$-equivalent if $H$ is continuous on the union of all the singular points and hyperbolic limit cycles of the vector field $v_0$.
	
	\item \emph{strongly topologically equivalent} provided that the map $H$ above is continuous.
	\end{itemize}
\end{defin}
 Weak equivalence is also called mild equivalence in some sources.

\begin{defin}
	A local family of \vfs  is called \emph{weakly structurally stable} if it is weakly topologically equivalent to any nearby family.
\end{defin}

\begin{thm} \label{thm:stab} A generic one-parameter unfolding of a generic \vf of class $PC$ is weakly structurally stable.
\end{thm}

Vector fields from this theorem have to satisfy an extra genericity assumption in addition to those included in the definition of class $PC$. This assumption is presented in Sec. \ref{subsec:nonsyn}, where an improved version of Theorem \ref{thm:stab} is stated.

The genericity assumption for the unfolding in Theorem \ref{thm:stab} is \emph{transversality to} $PC$. 

The previous theorem is wrong if the weak equivalence is replaced by the strong equivalence, see \cite{MP}.

\begin{rem}
 
Sing-equivalence has the following property. Let $V = \{v_\e\}$ and $W = \{w_\d\}$ be two sing-equivalent families.
For any singular point $O$ of $v_0$, let $O(\e )$ be a singular point of $v_\e $ depending continuously on $\e $ and such that $O(0) = O$. Put $\tilde O = H_0(O)$ and let $\tilde O(\delta)$ be a similarly defined singular point of $w_{\delta}$. Then
\begin{equation}\label{eqn:sing1}
  H_\e(O(\e)) = \t O(h(\e)).
\end{equation}
The same holds for limit cycles of $v_{\eps}$, $w_{\delta}$.

Sing-equivalence is designed to imply this property.

\end{rem}

\subsection{Time function}     \label{subsec:time}

The following arguments are based on the heuristic principle: local dynamics near an equilibrium point usually
determines a canonical chart at this point.

Let $v$ be a \vf of class $PC$, $\g $ be its parabolic limit cycle.
Let $\G $ be a cross-section to $\g $, $x$ a smooth chart on it with $x(\g \cap \G ) = 0$.
Let $P$ be a germ of the corresponding Poincaré map, $P(x) = x + ax^2 + \dots$. By assumption, $a \ne 0$. Rescaling $x$ and changing sign if needed we will make $a = 1$; so we will assume that 
\begin{equation}\label{eqn:pmap}
P(x)=x+x^2+\dots
\end{equation}

\begin{thm}    [Takens, \cite{T}]  \label{thm:pmap} Let $P $ be a $C^{\infty}$-smooth parabolic germ of the form \eqref{eqn:pmap}.  Then it has an infinitely smooth generator: there exists a germ of a vector field $u(x) = x^2 + \dots $ at zero,
whose time one phase flow transformation equals $P$:
$$
P = g^1_u.
$$
The smooth generator $u$ of $P$ is unique.
\end{thm}

Let $\G $ be a cross-section to $\g $, put $O = \G \cap \g$, and let $x$ be a chart on $\G $ with $x(O) = 0$ in which $P$ has the form \eqref{eqn:pmap}. Let $\G^+$ and $\G^-$ be the parts of $\G $ where $x \ge 0$ and $x \le 0$ respectively. Define the \emph{time functions} on $\G^+$ and $\G^-$, unique up to adding a~constant, in the following way. Choose two small numbers $b^-<0<b^+$,  and let $T^+(b)$ be the time of the motion from the point $b^+$ to the point $b\in \G^+ \setminus \{0\}$ along the solution of the equation $\dot x = u(x)$; let $T^-(b)$ be  the time of the  motion from the point $b^-$ to the point $b \in \G^- \setminus \{0\}$ along the solution of the equation $\dot x = u(x)$. In other words, 
$$
T^+(b) = \int^b_{b^+} \frac {dx }{u(x )}  \mbox{ for } b > 0, \ b \in \Gamma^+,
$$
$$
T^-(b) = \int^b_{b^-} \frac {dx }{u(x )}  \mbox{ for } b < 0,  \ b \in \Gamma^-.
$$

\subsection{Large bifurcation support}    \label{subsec: lbss}

The bifurcations in a local family that unfolds a \vf of class $PC$ are not only reduced to splitting and vanishing of the \lc $\g $. They also produce so called \emph{sparkling saddle connections} discovered by Malta and Palis \cite{MP}.

Suppose that the \vf $v \in PC $ has two saddles $E$ and $I$ on different sides of $\g $ whose separatrices wind
towards $\g $ in the positive and negative time respectively. The saddle $E$ lies outside, and the saddle $I$ inside $\g $; $E$ and $I$ stand for ``exterior'' and ``interior''.
Let $V = \{ v_\e \} $ be an unfolding of $v$ transversal to $PC $, $v_0 = v$,
$E(\e )$ and $I(\e )$ be the saddles of $v_\e $ continuous in $\e $ and \st $E(0) = E, I(0) = I$.
Let $\g $ disappear for $\e > 0$. Then \tes a sequence $\e_n \searrow 0$ \st the \vfs $v_{\e_n}$
have saddle connections between the saddles $E(\e_n), I(\e_n)$. These connections are called
\emph{sparkling saddle connections}.

This motivates the following definition.

\begin{defin}  \label{def:lbs}    Let $v \in PC $. \emph{The \lbs of} $v$ is the union of the parabolic cycle $\g $ and the closures of all the separatrices of the hyperbolic saddles that wind towards~$\g $ in the negative or positive time.
\end{defin}

\begin{rem}
Large bifurcation supports are defined in a much more general setting in \cite{I16}.
\end{rem}

The term is motivated by the heuristic statement that all the bifurcations that occur in the generic unfolding of $v$ are determined by those in a \nbd of the \lbs of $v$. For \vfs of class $PC$ this follows from Theorem~\ref{thm:class} below. In the general setting it is proved in \cite{GI}, work in progress.

The large bifurcation supports for the \vfs of class $PC$ are characterized by two so called marked finite sets on a circle.

\subsection{Marked finite sets}    \label{sub:marfin}

Large bifurcation supports may be rather complicated, see Figure \ref{fig:lbs}.

\begin{figure}
    \centering
\includegraphics[scale=0.5]{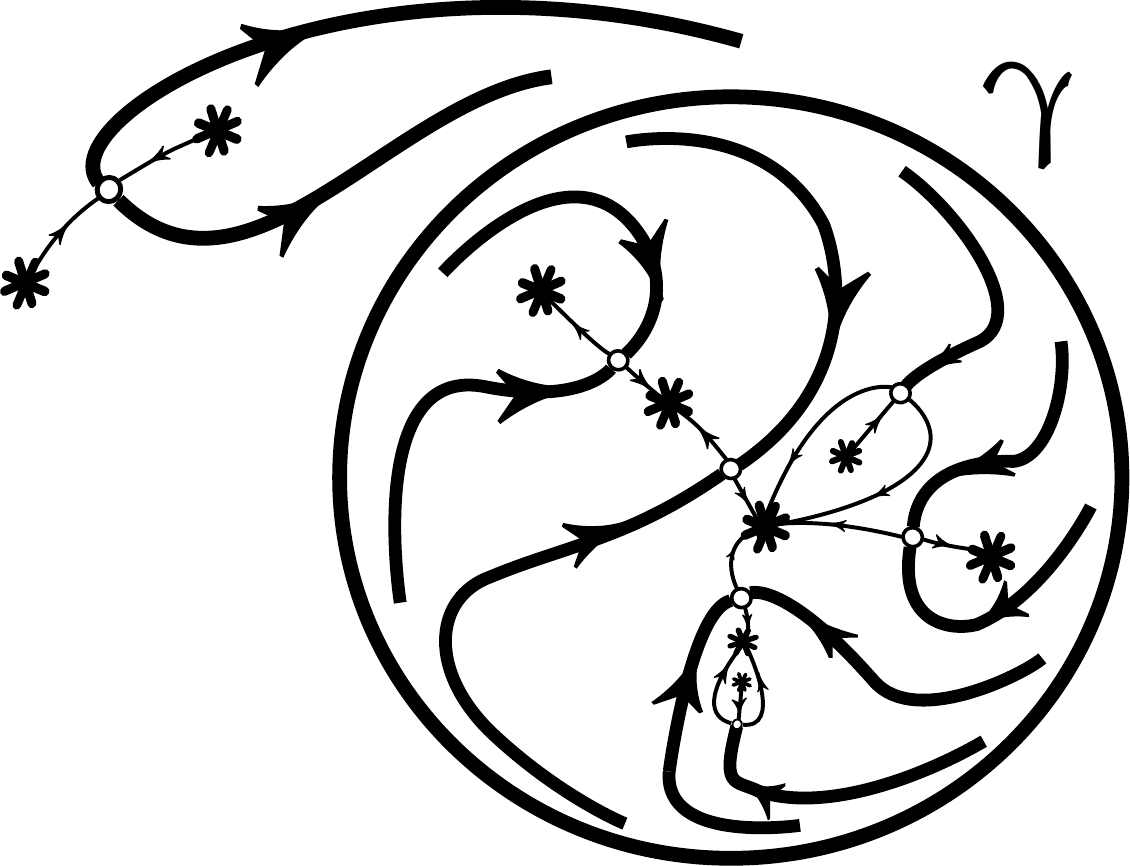}
\caption{Large bifurcation support of a $PC $ vector field (shown in thick curves). Here and below asterisks show sinks and sources of a vector field
}\label{fig:lbs}
\end{figure}

Yet they admit  a simple combinatorial description. The \pmap on $\G^-$, as well as on $\G^+$, in the charts $T^\pm $, is the mere translation by 1:
$$
T^\pm (P(x)) = T^\pm (x) + 1.
$$
Hence, the set of orbits of $P$ on $\G^\pm $ is a coordinate circle $S^1_\pm = \rr^+/\zz $, the coordinate is
$T^\pm \pmod\zz$. Note that this coordinate is defined uniquely up to an additive constant that depends on the choice of $b^{\pm}$ in Sec. \ref{subsec:time}.

Denote by $D^+$ the set of all intersection points of the
separatrices that wind toward $\g $ with the half open segment $[b^+, P(b^+)), \ b^+ \in \G^+$.
In the same way the set $D^-$ is defined for $b^- \in \G^-$. Let
\begin{equation}\label{eqn:apm}
A^\pm = T^\pm(D^\pm )\pmod\zz, \ A^\pm \subset S^1_\pm.
\end{equation}

Let us define the equivalence relations on $D^+$ and $D^-$. Namely, two points of $D^+$ (or $D^-$) are equivalent if they belong to the separatrices of the same saddle. This induces  equivalence relations  on $A^+$ and $A^-$. Note that any two equivalence classes $(a,b), (c,d)$ are not intermingled on the oriented circle: either both points $c,d$ belong to an arc from $a$ to $b$, or none of them.

\begin{defin}
\label{def:charset}
The equivalence relation on a finite set on a circle is called \emph{proper} if
each equivalence class consists of one or two points, and any two classes of two points each are not intermingled in the sense explained above. 

A finite set on a circle with a proper equivalence relation is called \emph{marked}.
\end{defin}
Thus for any \vf $v \in PC $ a pair of marked sets $A^\pm (v)$ on coordinate circles is defined.

\begin{defin}
The marked sets $A^\pm (v)$ are called the \emph{characteristic pair} (of sets) for the vector field $v\in PC$.
\end{defin}

\begin{rem}
Recall that the time coordinates on coordinate circles are defined modulo additive constants that depend on the choice of $b^{\pm}$ in Sec. \ref{subsec:time}.  So the characteristic sets $A^{\pm}(v)$ are defined modulo additive constants.
\end{rem}

\subsection{Non-synchronization condition}  \label{subsec:nonsyn}

\begin{defin}  \label{def:nonsyn} Two finite  sets $A^+,A^- \subset S^1$ are \emph{non-synchronized} provided that for any $\alpha \in \rr $,
\begin{equation}\label{eqn:nsync}
\# \left((A^+ + \alpha) \cap A^-\right) \le 1.
\end{equation}
\end{defin}

We can now give an explicit form of \autoref{thm:stab}.

\begin{thm}   \label{thm:stab1}  Suppose that characteristic sets $A^{\pm}(v)$ for  a vector field  $v \in PC$ are non-synchronized. Then any local one-parameter unfolding of $v$ transversal to the Banach manifold    $PC $ is structurally stable in the space of one-parameter families with the $C^1$ metric on it.
\end{thm}
This theorem is proved in Sec. \ref{sec-str}.

\begin{defin}
One-parameter local families described by this theorem are called \emph{$\pc $ families}.
\end{defin}

The following realization theorem holds.

\begin{thm}  \label{thm:real1}    Let $A^\pm $ be a pair of marked non-synchronized set on the coordinate circles. Then there exists a vector field $v$ of class $PC$ such that $A^{\pm}$ are characteristic sets of $v$.
\end{thm}
It is proved in Sec. \ref{subsec:mark}. The vector field set $v$ whose characteristic sets coincide with $A^{\pm}$ is in no way unique.

\subsection{Topologically equivalent \vfs of class $PC$ with non-equivalent one-parameter unfoldings}  \label{sub:noneq}

%


\begin{thm}   \label{thm:noneq}   There exist \vfs mentioned in the title of this section. More precisely, there exist topologically equivalent \vfs of class $PC$ whose generic one-parameter unfoldings are not sing-equivalent.
\end{thm}
We prove this theorem in Sec. \ref{sec-support}. 

We conjecture that this is the only result of this kind: bifurcations in other generic one-parameter unfoldings are determined by the topology of the phase portrait of an unperturbed vector field. 

\begin{conj}  \label{thm:eq}   Consider two generic one parameter local families  of \vfs $\{v_{\e}\}$ and $\{w_{\delta}\}$ \st $v_0, w_0$ are not of class $PC$. Suppose that $v_0, w_0$ are orbitally topologically equivalent. Then these local families $\{v_{\e}\}$ and $\{w_{\delta}\}$ are weakly equivalent.
\end{conj}


\subsection{Classification of $\pc$ families}   \label{subsec:class}

To any pair of finite sets on coordinate circles ordered counterclockwise and enumerated:
$$
A^+ = \{a_1^+,...,a_K^+ \}, \  A^- = \{a_1^-,...,a_M^- \}
$$
a set of pairwise differences corresponds:
\begin{equation}\label{eqn:diff}
 \tau_{km}:= \{a_k^+ - a^-_m\}, \quad  \Lambda (A^{\pm}) := \{\tau_{km} \mid k = 1, ... , K; m = 1,...,M\},
\end{equation}
where $\{a_k^+ - a^-_m\} \in [0,1)$ stands for the fractional part of $a_k^+ - a^-_m$; that is, $\tau_{km}$ is the length of the positively oriented arc from $a^-_m$ to $a_k^+$.
A pair $A^{\pm}$ is non-synchronized iff all the elements of the set $\Lambda (A^{\pm})$ are pairwise distinct.

\begin{rem}
 If one of two sets $A^-, A^+$ is empty, the set $\Lambda (A^{\pm})$ is empty.
\end{rem}
\begin{rem}
 If we add a shift to one of the sets $A^\pm$, then the shift is added to the set of differences $\Lambda(A^\pm)$: $\Lambda(A^{\pm})+\alpha=\Lambda(A^++\alpha, A^-)$. This is the reason for considering $\Lambda(B^\pm)+\alpha$ in the following definition.
\end{rem}

\begin{defin}  \label{def:seteq}   Two non-synchronized pairs of marked finite sets $A^{\pm}$ and $B^{\pm}$ on two circles are \emph{equivalent} if $|A^-|=|B^-|$, $|A^+|=|B^+|$, and for some shift $x\mapsto x+\alpha$, the sets
$ \Lambda (A^{\pm})$ and  $\{ \Lambda (B^{\pm}) +\alpha \}$ are ordered in the same way on $[0,1)$.
In more detail, let
$$
B^+ = \{b_1^+,...,b_K^+ \}, \  B^- = \{b_1^-,...,b_M^- \}
$$
be ordered counterclockwise, and put
$$
 \lambda_{km}:= \{b_k^+ - b^-_m\}, \quad  \Lambda (B^{\pm}) = \{ \lambda_{km} \mid  k = 1, \dots , K; m = 1,\dots,M\}.
$$
Then there exists $\alpha\in \bbR$ such that
\begin{equation}\label{eqn:order5}
  \tau_{km} > \tau_{k'm'} \Rightarrow \{\lambda_{km}+\alpha\} > \{\lambda_{k'm'}+\alpha\}.
\end{equation}

\end{defin}
We will use this definition for characteristic pairs of sets, see Definition \ref{def:charset} of Sec. \ref{sub:marfin}. Recall that the characteristic sets $A^{\pm}(v)$ for a vector field $v$ are well-defined up to  additive constants, so the equivalence of characteristic sets for two vector fields is well-defined.

\begin{defin}
\label{def:seteq:char}
 Let two vector fields $v,w$ of class $PC$  be orbitally topologically equivalent. Enumerate the sets $A^{\pm}(v)$ counterclockwise along coordinate circles. This enumeration and the orbital topological equivalence of $v,w$ induce the enumeration of the sets $\tilde A^{\pm}(w)$: the intersections of transversals with corresponding separatrices of $v,w$ will have the same numbers.

 Now, the vector fields $v$ and $w$ are said to have  \emph{non-synchronized and equivalent characteristic sets} if $A^{\pm}(v)$ and $A^{\pm}(w)$ are non-synchronized and equivalent in the sense of Definition \ref{def:seteq} (with the numbering described above).
\end{defin}

\begin{thm}  \label{thm:class}  1. Let two $\pc$ families be sing-equivalent. Then their characteristic sets on the coordinate circles are equivalent in the sense of Definition \ref{def:seteq:char}.

2. Let two \vfs of class $PC$ be orbitally topologically equivalent. Let their characteristic pairs be non-synchronized and equivalent in the sense of Definition \ref{def:seteq:char}. Then the generic unfoldings of these two \vfs are sing-equivalent.
\end{thm}

Any pair of non-synchronized marked finite sets determines the bifurcation scenario (sequence of bifurcations) in the corresponding class of $\pc$ families. This scenario will be explicitly described, see Sec. \ref{sub:scen}.

\section{Bifurcations in the $\pc$ families}

\subsection{Embedding theorem for families}    \label{subsec:emb}   Takens embedding Theorem \ref{thm:pmap} for parabolic germs may be extended to their unfoldings.

\begin{thm}   \label{thm:IYa}  \cite{IYa} Let $P_\e$ be a generic one-parameter $C^\infty $ unfolding of a parabolic germ
$$
P_\e (x) = x + x^2 + \e + \dots .
$$
Then in the domain $\{ \e \ge 0\} \setminus \{ 0,0\} $, the family $P_\e $ is $C^\infty $ equivalent to the time one
phase flow transformation of the field
\begin{equation}\label{eqn:gen}
u_\e (x) = \frac {x^2+\e }{1+a(\e )x},
\end{equation}
where $a(\e )$ is a $C^\infty $ function; the equivalence is infinitely smooth both in $x$ and $\e $.
\end{thm}

The coordinate $x_\e$ that brings $P_\e$ to the time-one shift of the \vf $u_\e$ is called \emph{normalizing.}
From now on, the coordinate on the cross-section $\G$ for $\e \ge 0$ is the normalizing coordinate $x_\e$.

\begin{rem}\label{rem:bound}
Since the normalizing coordinate $x_\e$ is $C^{\infty}$ smooth on the set
$\{ \e = 0\} \setminus \{ 0,0\} $, it may be smoothly extended to some \nbd of any point of this set. As a corollary, all the derivatives of $x_\e$ at $\e = 0, x_\e \not = 0 $ exist and are finite.
\end{rem}

\subsection{Transversal loops and canonical coordinates on them}   \label{subsec:non}

Consider a \vf $v$ of class $PC$; let $\g $ be its parabolic cycle, $\G $ be a cross-section to $\g , \ O = \G \cap \g $, and
$$
P: (\G ,O) \to (\G ,O)
$$
be the germ of the Poincaré map corresponding to $\g$. Consider a generic unfolding $V = \{ v_\e \} $ of $v, \ v_0 = v$. Let $P_\e $ be the
corresponding \pmap of $v_\e $, and $x_\e$ the corresponding normalizing chart on $\G$ provided by Theorem \ref{thm:IYa}.
Let $C^+$ and $C^-$ be two \nccs around $\gamma$, $C^\pm \cap \G    = b^\pm $. For $\e>0$, the cycle $\gamma$ vanishes, and the Poincaré map $\Delta_{\eps} \colon C^-\to C^+$ is well-defined. We will now choose coordinates $\ph^{\pm}_\e$ on $C^{\pm}$ such that $\Delta_{\eps}$ becomes a rotation in these coordinates.

For $\e > 0$ consider a one form $\om_\e $ on $\G $ dual to the \vf $u_\e $:
$$
\om_\e = \frac {dx_\e}{u_\e(x_\e )}.
$$
For small $b \in \G $ and $\e > 0$ let
\begin{equation}\label{eqn:te}
T^\pm_\e (b) = \int^{x_\e (b)}_
{x_\e (b^\pm)} \om_\e .
\end{equation}
Let
$$
\tau (\e ) = T^-_\e (b^+).
$$
This function may be explicitly calculated:
\begin{equation}
\label{eq:tau}
\tau (\e ) = \frac {1}{\sqrt \e }(\arctan \frac {x_\e (b^+)}{\sqrt \e } - \arctan \frac {x_\e (b^-)}{\sqrt \e }) +
\frac {a(\e )}{2}\log \frac {x^2_\e (b^+)+\e }{x^2_\e (b^-)+\e }.
\end{equation}
Note that for $\e > 0$,
$$
T^-_\e (b) = T^+_\e (b) + \tau (\e ).
$$
Note that  formula~\eqref{eqn:te} works for $\e = 0$ also, with the following restriction:
$$
T^-_0 = T^- \mbox{ on } \G^- \setminus \{ 0\} ;
$$

$$
T^+_0 = \lim_{\e \to 0} T^-_\e  - \tau (\e ).
$$
The time functions $T^\pm_\e $ induce coordinates $\ph^\pm_\e $ on $C^\pm $ in the following way. Consider first $C^-$.
Take a point $a \in C^-$ and emerge a forward orbit of $v_\e $ from it, see Figure \ref{fig:coord}. Let $b \in \G^-$ be its first intersection point with $\G $. Take
$$
\ph^-_\e (a) = T^-_\e (b).
$$
Note that $\ph^-_\e (b^-) = 0$; as $a$ tends to $b^{-}$, one of the one-sided limits of $\ph^-_\e$ at $b^-$ is $0$ and the other is $1$. Thus $\ph^-_\e $ maps $C^-$ onto the coordinate circle $S^1_-$. The same construction provides a function
$$
\ph^+_\e: C^+ \to S^1_+ , \ a' \mapsto T^+_\e (b'),
$$
see Figure \ref{fig:coord}.
These $\e $-dependent coordinates $\ph^\pm _\e $ on the transversal loops $C^\pm $ are called \emph{canonical}.

Without loss of generality we may assume that the cycle $\g $ is time oriented clockwise. Then the transversal loops $C^\pm $ are oriented counterclockwise by the canonical coordinates.

\subsection{ The \pmap of the \nccs }

Consider a small $\e > 0$. The orbit of the \vf $v_\e$ that starts at a point $a \in C^-$ eventually reaches $C^+$ at a unique point $a'$. This defines the \pmap
$$
\Delta_\e : C^- \to C^+, \ a \to a'
$$
along the orbits of $v_\e $, see Figure \ref{fig:coord} again.

\begin{figure}
    \centering
\includegraphics[scale=0.3]{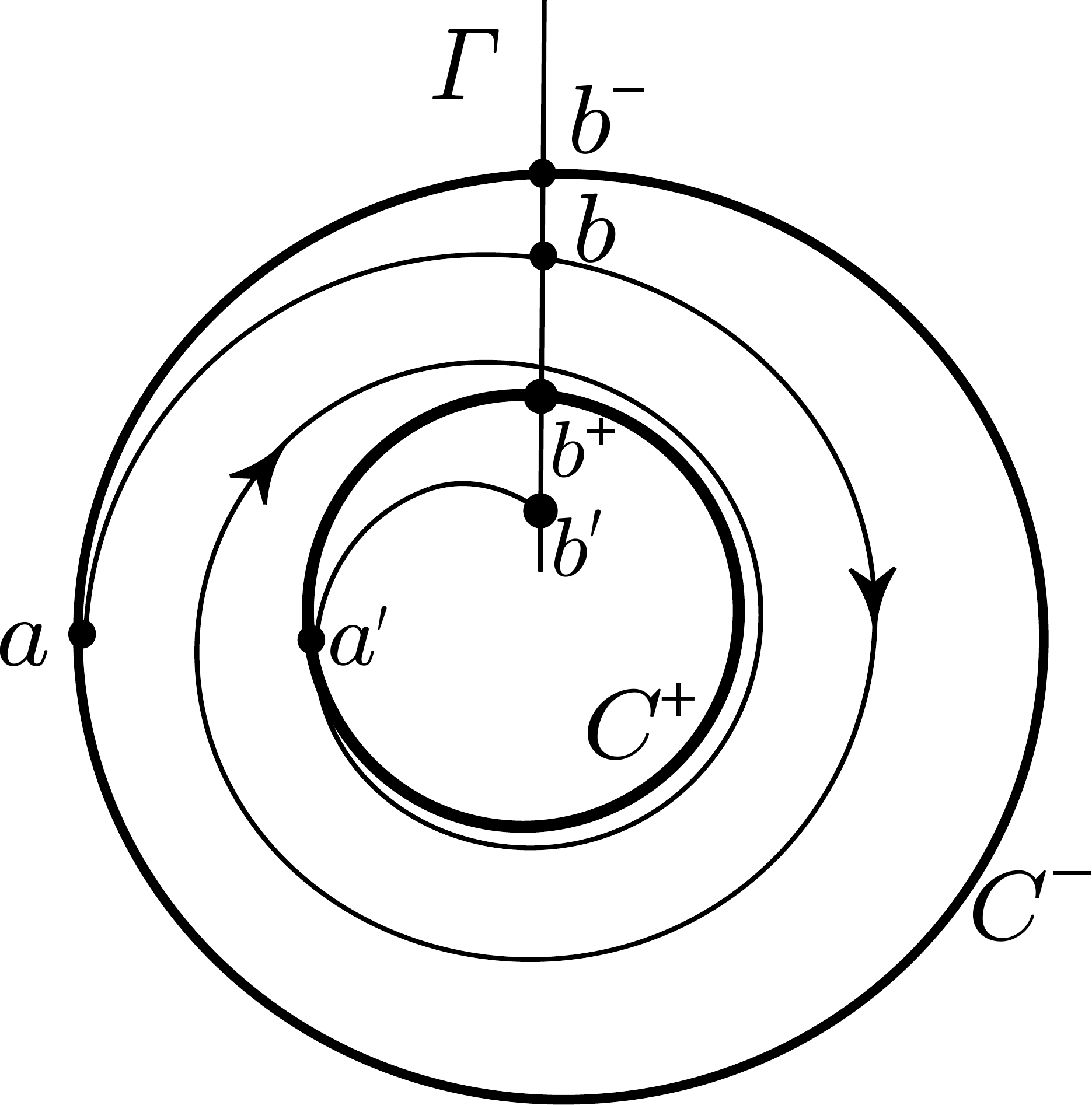}
\caption{Canonical coordinates and the Poincaré map on the transversal loops }\label{fig:coord}
\end{figure}

\begin{prop} \label{prop:pmap}
In the coordinates $\ph_\e^{\pm}$, the \pmap $\Delta_\e : C^- \to C^+$ is a mere rotation:
\be \label{eqn:rot}
\ph_\e^+ (\Delta_\e (a)) = \ph_\e^- (a) - \tau (\e )\pmod\zz.
\ee
\end{prop}
In what follows, the map $\Delta_{\e}$ in $\ph_\e^{\pm}$-coordinates (i.e. the rotation by $-\tau(\e)$) will be denoted by the same symbol.

\begin{proof}
Let $a \in C^-$, and $b$ be the same as above, that is, the first intersection point with $\G $ of the
forward orbit of $v_\e $ emerging from $a$. Let $a' \in C^+$ be the image of $a$: $a' = \Delta_\e (a)$, and $b'$ be the first intersection point with $\G $ of the forward orbit of $v_\e $ that emerges from $a'$. Then, by definition of canonical coordinates,
$$
\ph^-_\e(a) = T^-_\e (b), \ \ph^+_\e(a') = T^+_\e (b').
$$
On the other hand,
$$
b' = P^n_\e (b)
$$
for some $n$. Hence,
$$
T^-_\e (b') - T^-_\e (b) = n,
$$
because in the chart $T_\e $ the \pmap $P_\e $ is a mere shift by $1$. Moreover,
$$
T^+_\e (b') = T^-_\e (b') - \tau (\e ).
$$
Recall that $a' = \Delta_\e (a)$. Hence,
$$
\ph^+_{\e}(\Delta_\e (a)) = \ph^+_\e(a') = T^+_\e (b') = T^-_\e (b') - \tau (\e )  = T^-_\e (b) - \tau (\e ) \pmod\zz = \ph^-_\e (a) - \tau (\e )\pmod\zz,
$$
see Figure \ref{fig:coord}.
This proves the proposition.
\end{proof}

\subsection{Characteristic sets on the transversal loops}  \label{subsec:mark}
Let $S^-=\{s_m^-\}$ be the set of all intersections of $C^-$ with separatrices of $v$, enumerated counterclockwise along $C^-$.
Let $l_{m}^-$ be the corresponding separatrices and $E_m$ be the corresponding saddles of $v$.
The canonical coordinate $\ph^{-}_0$ maps the set $S^-$ to the characteristic set
$$
A^- = \{a^-_1, ... , a^-_M \}, \ a^-_m = \ph^{-}_0 (s^-_m),
$$
see Figure \ref{fig:charset}.

If two points $a_m^-=\ph^{-}_0(s_m^-)$ and $a_{m'}^-=\ph^-_0(s_{m'}^-)$ are equivalent as the points of the marked set $A^-$, then $E_m = E_{m'}$.

Similarly, separatrices $l_k^+$ of saddles $I_k$ are all separatrices of $v$ that intersect $C^+$, $S^+ = \{s_{k}^+\}$ are intersection points, and $a_k^+ = \ph_0^+(s_k^+)$.
This determines another characteristic set
$$
A^+ = \{a^+_1, ... , a^+_K \}, \ a^+_k = \ph^{-}_0 (s^+_k).
$$



\begin{figure}
 \begin{center}
\includegraphics[width=0.55\textwidth]{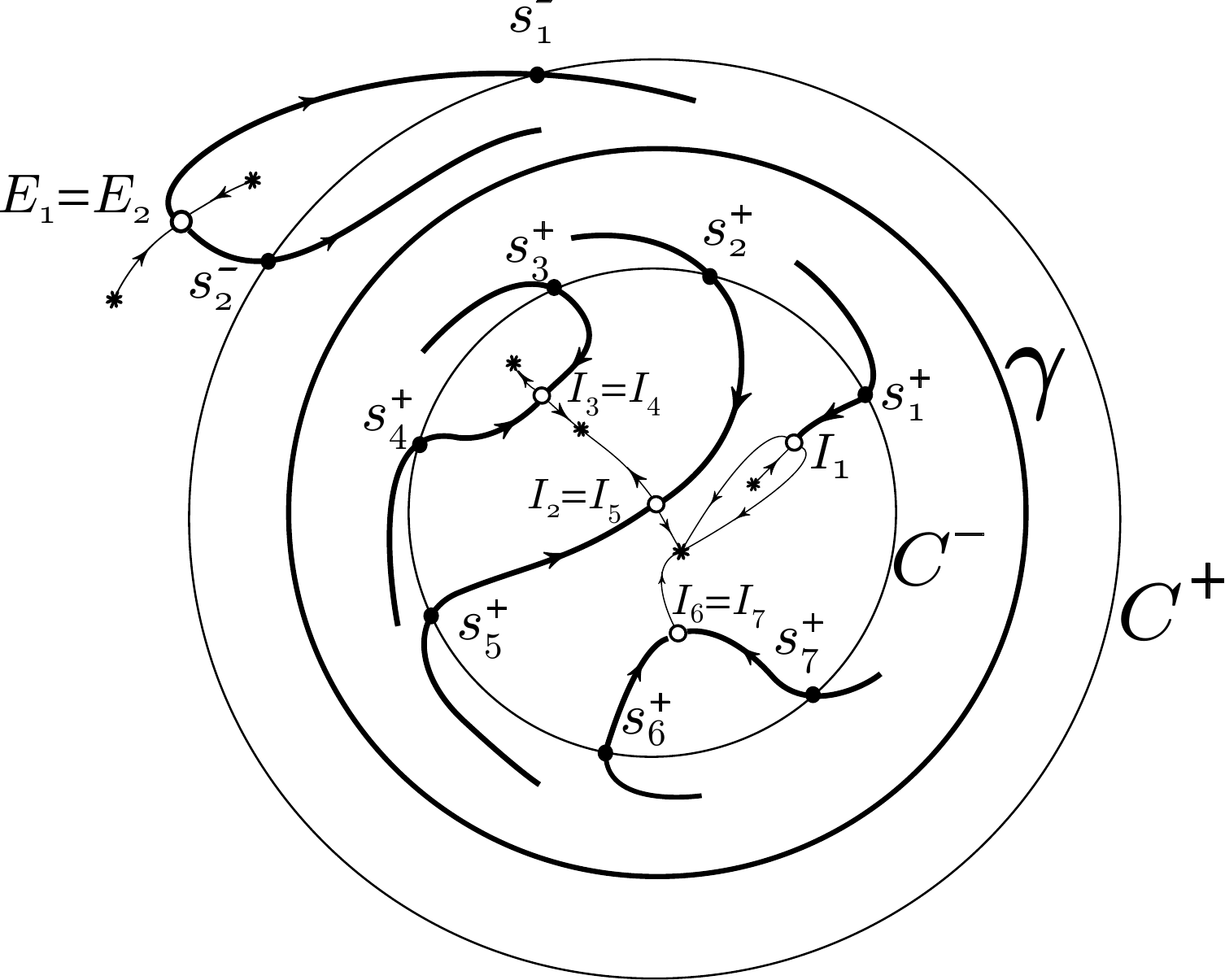}
 \end{center}
 \caption{Characteristic sets on the transversal loops for the $PC$ vector field}\label{fig:charset}
\end{figure}

\subsection{The connection equation}   \label{subsec: conn}

In this and the next sections we describe the bifurcations in $\pc$ families.

Let $V = \{v_\e\}$ be a $\pc$-family.
The \vfs $v_\e$ have saddles $E_m(\e), \ I_k (\e )$ continuously depending on $\e, \ E_m(0) = E_m, \ I_k(0) = I_k$.
The germs of their separatrices at these saddles depend continuously on $\e $.
Denote by $l_m^-(\e )$ the separatrix with a germ $(l_m^-(\e ), E_m(\e ))$ that is
continuous in $\e $ and coincides with $(l_m^-, E_m)$ for $\e = 0$. In the same way the separatrices $l_k^+(\e )$ of the saddles $I_k(\e )$ are defined. Let $s_m^-(\e )$ be the (unique for $\e $ small) intersection of
$l^-_m(\e )$ and $C^-$; let $s_k^{+}(\e)$ be the intersection of $l_k^+(\e )$ and~$C^+$. Define $S^\pm(\e) = \{s_{i}^\pm(\e)\}$.  Put
$$
a_m^-(\e) := \ph_\e^-(s_m^-), \quad a_k^+(\e) := \ph_\e^+(s_k^+),
$$
and let $A^-(\e ) := \{a_m^-(\e)\}$, $A^+(\e):=\{a_k^+(\e)\}$. The sets $A^{\pm}(\e )$ depend continuously on $\e $ and coincide with $A^\pm $ for $\e = 0$.
In particular, if $A^\pm $ is a pair of \ns sets, then $A^\pm (\e )$ also is for $\e $ small.

Let
\be  \label{eqn:tkm}
\tau_{km}(\e ) = \{a^+_k(\e ) -  a^-_m(\e )\}; \ \tau_{km}(\e ) \in [0,1).
\ee
If none of $\tau_{km}$ is equal to zero, $\tau_{km}(\e)$ depends continuously on $\e$ and coincides with $\tau_{km}$ for $\e=0$.
Without loss a generality, we may and will assume that none of $\tau_{km}$ is $0$. Elsewhere, we will slightly change $b^-$ keeping $b^+$ unchanged. This will rotate $A^-$ and preserve $A^+$, thus change $\tau_{km}$ by an additive constant.

As the pair $A^\pm (\e )$ is non-synchronized, the values of $\tau_{km}(\e )$ are pairwise distinct.

A saddle connection between the saddles $E_m(\e )$ and $I_k(\e )$ occurs iff
$$
a_k^+(\e ) = \Delta_\e (a_m^-(\e )).
$$
By Proposition~\ref{prop:pmap}, this is equivalent to
$$
a_k^+(\e ) = a_m^-(\e ) - \tau (\e ) \pmod\zz
$$
or equivalently,
$$
\tau_{km} (\e ) = -\tau (\e ) \pmod\zz.
$$
Another form of this equation is:
\begin{equation}\label{eqn:conn}
\tau_{km}(\e ) = -\tau (\e ) + n, \ n \in \nn .
\end{equation}
This is a \emph{connection equation}.

\begin{prop}   \label{prop:conn} For $\e $ small and $n$ large enough, equation~\eqref{eqn:conn} has a unique solution $\e = \e_{kmn}$.
\end{prop}
When $\e=\e_{kmn}$, the vector field $v_{\e}$ has a  separatrix connection between the saddles $E_k(\e)$ and $I_m(\e)$. This separatrix connection makes $n$ winds around $\g $ between $C^-$ and $C^+$.

\begin{proof}   Take $\e_0 > 0$ so small that the function $\eps \mapsto \tau_{km}(\e)$ is well-defined on $[0,\e_0]$. By Theorem~\ref{thm:IYa}, this function has a bounded derivative on the whole segment. On the other hand, $\tau' (\e ) \to -\infty $ as $\e \to 0$. Indeed, by~\eqref{eq:tau}
$$
\tau (\e ) = \left.\strut F(x,\e ) \right|^{\textstyle x = x_\e(b^+)}_{\textstyle x = x_\e(b^-)}
$$
where
$$
F(x,\e ) = \frac {1}{\sqrt \e }\arctan \frac {x}{\sqrt \e } + \frac {a(\e )}{2}\log (x^2 + \e ).
$$
Then
$$
\tau'(\e ) = F_\e (x,\e )|^{x = x_\e(b^+)}_{x = x_\e(b^-)} + F_x(x_\e (b^+), \e ) \cdot D_\e x_\e (b^+) - F_x(x_\e (b^-), \e ) \cdot D_\e x_\e (b^-).
$$
For any $b \ne 0$, the functions $F(b,\e )$ and $x_\e (b)$ are well defined in a \nbd of $\e = 0$ and have bounded derivatives. Hence, the second and the third terms in the expression for $\tau' (\e )$ are bounded near $\e = 0$. On the other hand,
$$
F_\e (x,\e ) = -\frac {1}{2\e^{3/2}}\arctan \frac {x}{\sqrt \e }  - \frac {1}{2\e (x^2+\e )}+\dots
$$
dots stand for bounded terms. Hence,
$$
F_\e (x,\e )|^{x = x_\e(b^+)}_{x = x_\e(b^-)} \to -\infty \mbox{ as } \e \searrow 0.
$$
Take $\e $ so small that
$$
(\tau (\e ) + \tau_{km}(\e ))' < 0.
$$
Then for any $n > \tau (\e_0) + \tau_{km}(\e_0)$, the connection equation~\eqref{eqn:conn} has a unique solution $\e_{kmn} < \e_0$.
\end{proof}

\subsection{The bifurcation scenario} \label{sub:scen}

The bifurcation scenario in a one-parameter family is a sequence of bifurcations that occur as $\e $ changes. The previous section shows that for $\e>0$, the sparkling saddle connections between saddles $E_{k}(\e)$ and $I_m(\e)$ occur for $\e=\e_{kmn}$.

Each bifurcation of the sparkling saddle connection goes in the same way as for usual saddle connections: the incoming separatrix of one saddle changes its $\a$-limit set, and the outgoing separatrix of another saddle changes its $\omega$-limit set, see Figure \ref{fig:bif-connect-a}. The only additional feature of the sparkling saddle connection is that for the critical parameter value $\e=\e_{kmn}$ the connection between the saddles  $E_m(\e)$ and $I_k(\e)$ winds around $\gamma$ many times, see Figure \ref{fig:bif-connect-b}. Yet topologically these pictures are the same: the second one may be transformed to the first one by the iterated Dehn twist.

\begin{figure}
\begin{center}
\centering
\subcaptionbox{\label{fig:bif-connect-a}}{\includegraphics[width=0.17\textwidth]{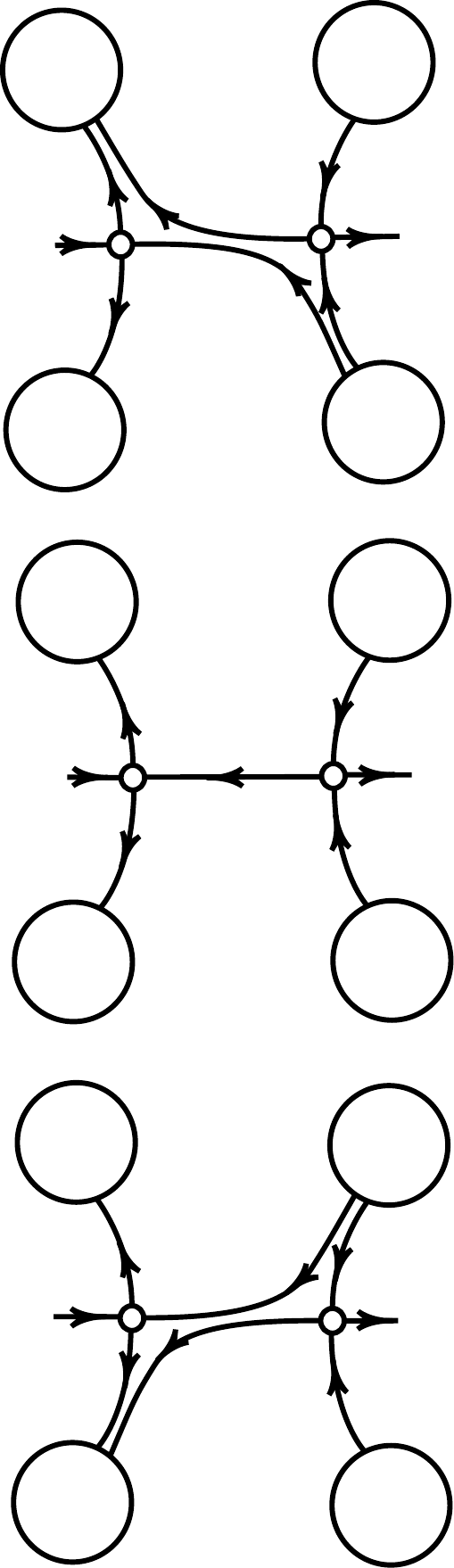}}
 \hfil
\subcaptionbox{\label{fig:bif-connect-b}}{\includegraphics[width=0.4\textwidth]{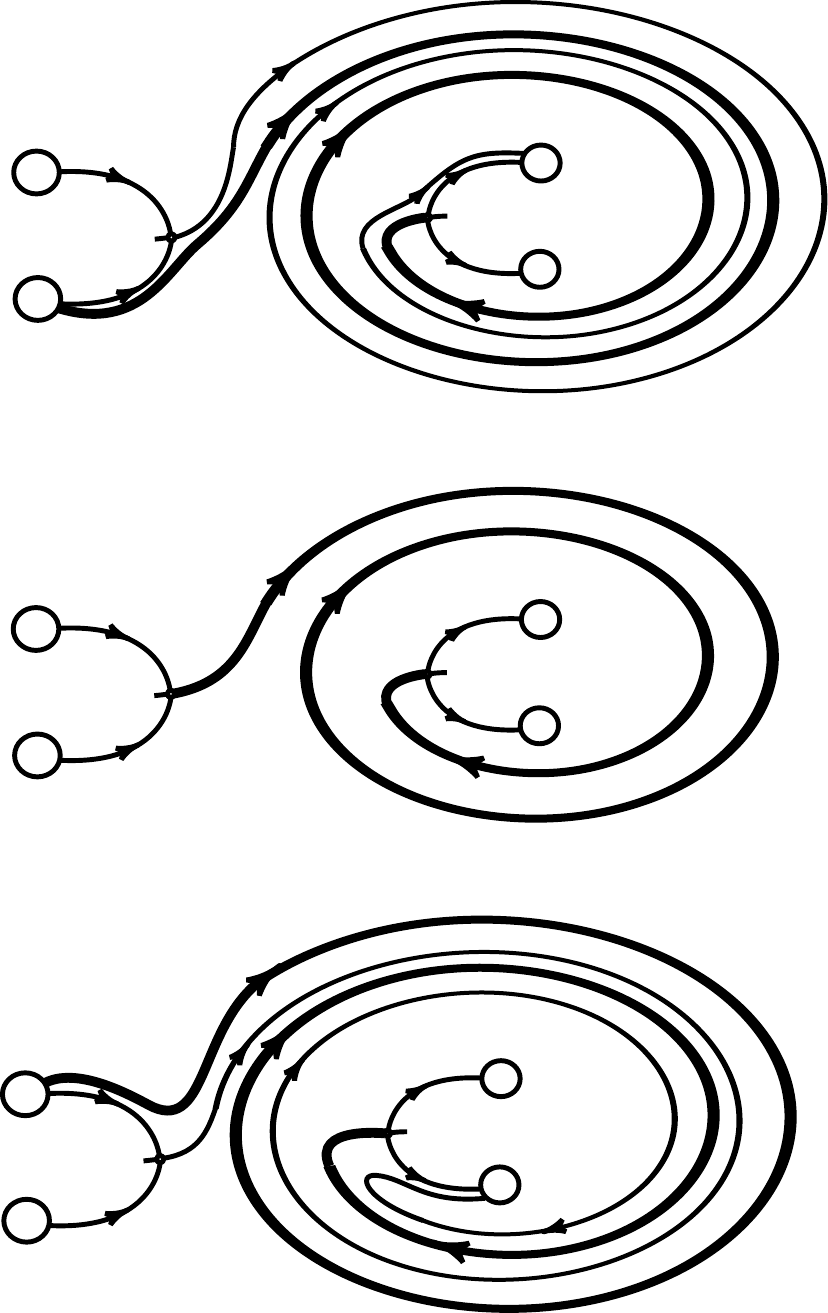}}
\caption{The bifurcation of a saddle connection (left) and the bifurcation of the sparkling saddle connection (right). Large circles are \nccs around $\alpha$-and $\omega$-limit sets of saddle separatrices, small circles are saddles}\label{fig:bif-connect}
  \end{center}
  \end{figure}

To finish the description of the bifurcation scenario, we need to describe the order in the set $\{\e_{kmn}\}$.
Recall that we assume without loss a generality that none of  $\tau_{km}$ is equal to zero.
The order of numbers $\e_{kmn}$ is described by the following proposition.
\begin{prop}\
\label{prop-order-e}
For sufficiently large $n$ and any $k,m,k',m'$, we have $\e_{kmn}< \e_{k'm'(n-1)}$.

For sufficiently large fixed $n$, the order of $\e_{kmn}$ does not depend on $n$ and coincides with the order of $\tau_{km}$. In more detail, if for some natural $k,m,k',m'$, we have $0<\tau_{km} < \tau_{k'm'}<1$, then for sufficiently large $n$, $\e_{kmn} < \e_{k'm'n}$.
\end{prop}
\begin{proof}
The first statement follows directly from the connection equation \eqref{eqn:conn}, the inequality $0<\tau_{km}(\e)<1$ and monotonicity of $\tau$: this function tends monotonically to infinity as positive $\e$ decreases.

Since $\tau_{km} =\tau_{km}(0)$ and  $\tau_{km}(\e)$ depends continuously on $\e$,  we have $\tau_{km}(\e_{kmn}) < \tau_{k'm'}(\e_{k'm'n})$ for large $n$. Now \eqref{eqn:conn} implies $-\tau(\eps_{kmn}) < -\tau(\eps_{k'm'n})$. However, the function $\tau$ decreases in a small neighborhood of zero, so $\eps_{kmn}< \eps_{k'm'n}$.
\end{proof}

 Recall that the numbers $\tau_{km}$ are well-defined modulo an additive constant that depends on the choice of $b^{\pm}$, and the set of bifurcation parameters $\{\e_{kmn}\}$ does not depend on this choice. However, that does not contradict the proposition above, because when we change $b^{\pm}$, the same bifurcation value $\e=\e_{kmn}$ will obtain a different  number $n$. Here we describe this change in more detail.

\begin{defin}
 Let $\e_{kmn}$ be as above. Let $i, 1<i<K$, and $j, 1<j<M$ be two indices, $N$ be an integer number.

 \emph{A cyclical shift of $n$} in the set $\e_{kmn}$  is the change of numeration: the bifurcation parameter $\e_{kmn}$ obtains indices $kmn'$, where $n'= n+N$ for $\e_{kmn}< \e_{ijn}$ and $n'=n+N-1$ otherwise.
\end{defin}

Suppose that we change our choice of $b^{\pm}$ replacing these points by $\tilde b^{\pm}$. The number $\tau_{km}(\e)$ is replaced by $\tilde \tau_{km}(\e) = \{\tau_{km}(\e) +\alpha(\e)\}$, where $\alpha(\e) =   T^-_{\e}(\tilde b^-) - T^+_{\e}(\tilde b^+) $; the charts $T^{\pm}(\e)$ correspond to $b^{\pm}$, see \eqref{eqn:te}. We assume that none of $\tilde \tau_{km}(0)$ is zero.
Similarly, $\tau(\e)$ is replaced by $\tau(\e)-\alpha(\e)$.

 \begin{prop}
 \label{prop-order-e-b}
The change of  $b^{\pm}$ described above results in a cyclical shift of $n$ in the set $\{\e_{kmn}\}$; any cyclical shift may be achieved.
 \end{prop}

 \begin{proof}
 Recall that $\e_{kmn}$ is the solution of the connection equation $\tau_{km}(\e ) = -\tau (\e ) + n$. Then it also solves the equation $\tau_{km}(\e )+\alpha(\e) = -\tau (\e )+\alpha(\e) + n$, i.e.
 $$
 \{\tau_{km}(\e )+\alpha(\e)\} = -(\tau (\e )-\alpha(\e)) + n - [\tau_{km}(\e )+\alpha(\e)].
 $$
 This is a connection equation in new coordinates, but $n$ is replaced by $n-[\tau_{km}(\e )+\alpha(\e)]$. So the bifurcation parameter $\e_{kmn}$ will obtain another number $n' = n-[\tau_{km}(\e )+\alpha(\e)]$.

 Note that the integer number $[\tau_{km}(\e )+\alpha(\e)]$ does not depend on $\e$ for small $\e$; this follows from the fact that none of $\tilde \tau_{km}(0)$ is zero. However it may depend on $k,m$: if it equals $N$ for $0< \tau_{km}(\e) < \{1-\alpha(\e)\}$, then it equals $N+1$ for $\{1-\alpha(\e)\}< \tau_{km}(\e)<1$. These inequalities on $\tau_{km}(\e)$ are equivalent to $\e_{kmn}< \e_{ijn}$ and $\e_{kmn}\ge \e_{ijn}$ for some $i,j$, because the order in the set $\{\e_{kmn}, n \text{ fixed}\}$ is the same as the order of $\tau_{km}$; so we have a cyclical shift of $n$ in the set $\{\e_{kmn}\}$. One can easily show that the choice of $\tilde b^{\pm}$ may give any prescribed value of $\alpha(0)$, thus it may result in arbitrary cyclical shift of $\{\e_{kmn}\}$. This completes the proof.
\end{proof}

%
%


We conclude that the bifurcation scenario in  $\pc$ families repeats cyclically as $\e\to 0$. As $\e$ decreases, in the family $v_{\e}$ we have several saddle connections that make $n$ winds in a small neighborhood of  $\gamma$ between $C^-$ and $C^+$,  then several saddle connections that make $n+1$ winds, etc. The $n$-wind saddle connections occur in one and the same order for all $n$, and this order is determined by the order of $\tau_{km}$. However, the number $n$ of winds of a particular saddle connection depends on the choice of $b^-, b^+$.

One of our main goals is achieved: the bifurcation scenario is described and justified.

Note that the bifurcating separatrices of the \vfs $v_\e$ are located not only inside a small \nbd of the large bifurcation support (see Definition \ref{def:lbs}), but also outside it. Yet this bifurcation is predicted by what happens in a \nbd of the large bifurcation support.

\subsection{Realization Lemma for discs}

In this and the next sections Theorem \ref{thm:real1} is proved, skipping the routine details.

Begin with a Realization Lemma that will be used not only here but in the study of the global bifurcations of the \vfs with a separatrix loop. Consider a Morse -- Smale \vf in a disc $D$ with the boundary $C$ transversal to $v$. No special coordinate on $C$ is considered. Let $A$ be the set of the intersection points of the separatrices of $v$ with $C$; two points of $A$ are equivalent iff they belong to the separatrices of the same saddle. Thus $A$ is a marked set. Let us call it \emph{the characteristic set} of $v$.

\begin{lem}
  \label{lem:marked} Consider a marked set $A$ on a circle $C$ that is a boundary of a disc $D$. Then \tes a $C^{\infty}$ \vf $v$ in $D$ \st $A$ is a characteristic set of $v$.
\end{lem}

\begin{figure}
    \centering
    \subcaptionbox{\label{fig:realiz-a}}{\includegraphics[scale=0.6]{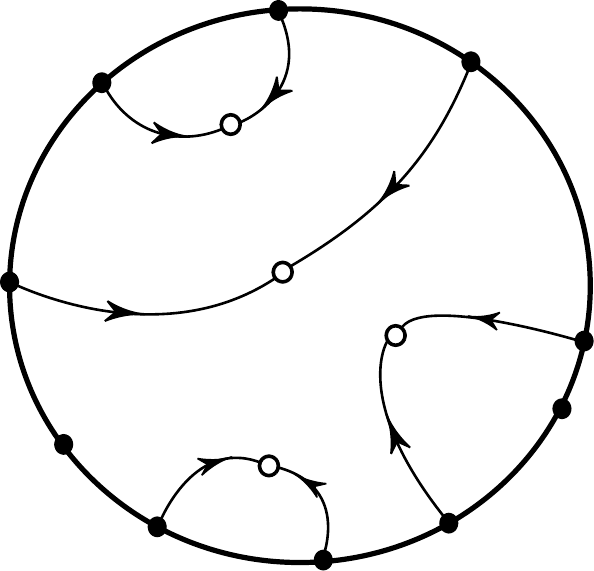}}
    \hfil
    \subcaptionbox{\label{fig:realiz-b}}{\includegraphics[scale=0.6]{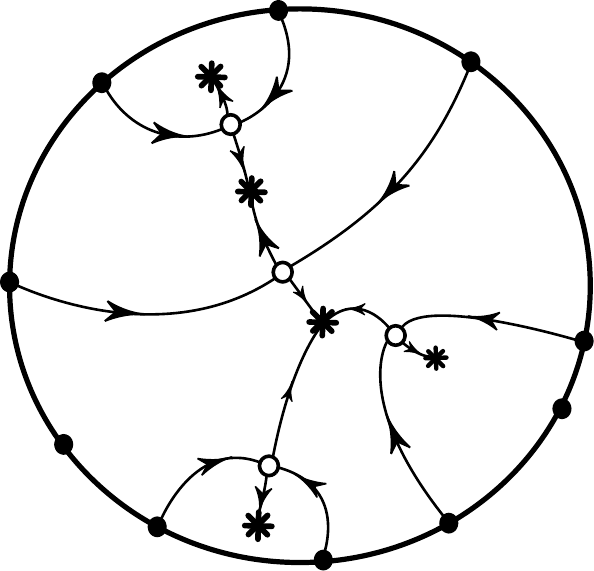}}
    \hfil
    \subcaptionbox{\label{fig:realiz-c}}{\includegraphics[scale=0.6]{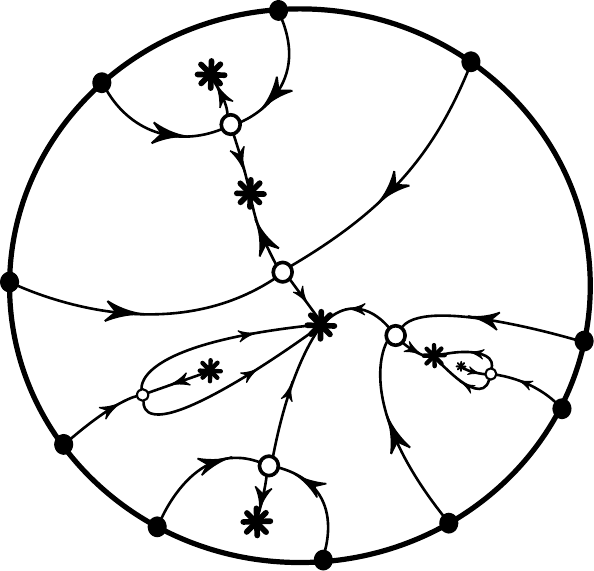}}
\caption{Realization of characteristic sets: steps a, b, c.} \label{fig:real}
\end{figure}


\begin{proof}
Without loss of generality, we assume that the vector field $v$ on $C$ points inside $D$.

$a$. Take any two equivalent points of the set $A$ and connect them by a smooth simple arc transversal to $C$ at its endpoints, whose interior part lies in $D$. The arcs should be pairwise disjoint. Take a point in each arc different from its endpoints; this will be a saddle of the \vf $v$ to be constructed. The parts of the arcs from the endpoints  to the saddles should be arcs of the \emph{incoming separatrices of} $v$, see Figure~\ref{fig:realiz-a}.
This operation is proceeded for all the pairs of equivalent points of $A$.

$b$. The disc $D$ is split by the arcs just constructed to topological disks. Let us choose a point inside each of these disks; this will be an attractor of $v$. Connect each saddle on the boundary of the topological disk above to the attractor inside it; this will be an outgoing separatrix of the saddle, see Figure~\ref{fig:realiz-b}. Now each saddle has four separatrices: two ingoing and two outgoing.

$c$. The domains to which $D$ is split have one of the two shapes shown on Figure~\ref{fig:shapes}.

\begin{figure}
    \centering
\includegraphics[scale=0.5]{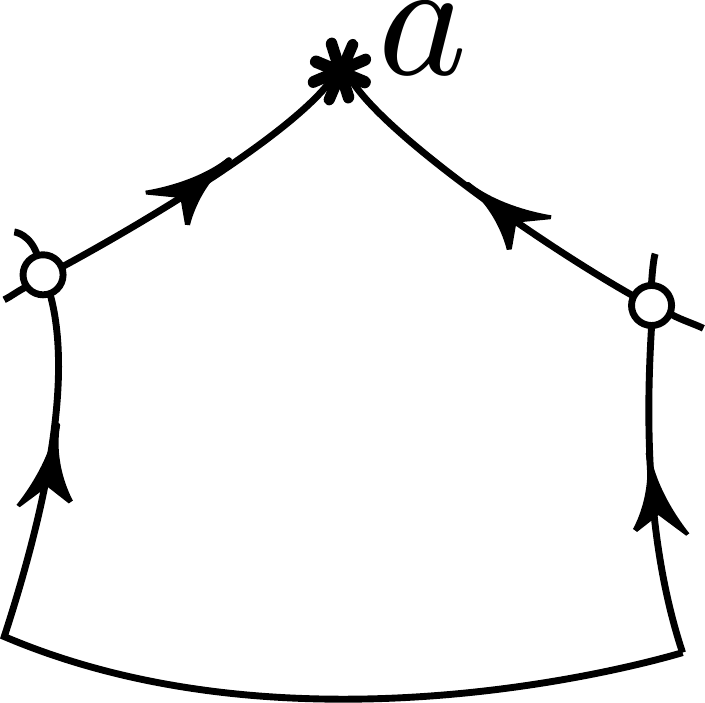}
\hfil
\includegraphics[scale=0.5]{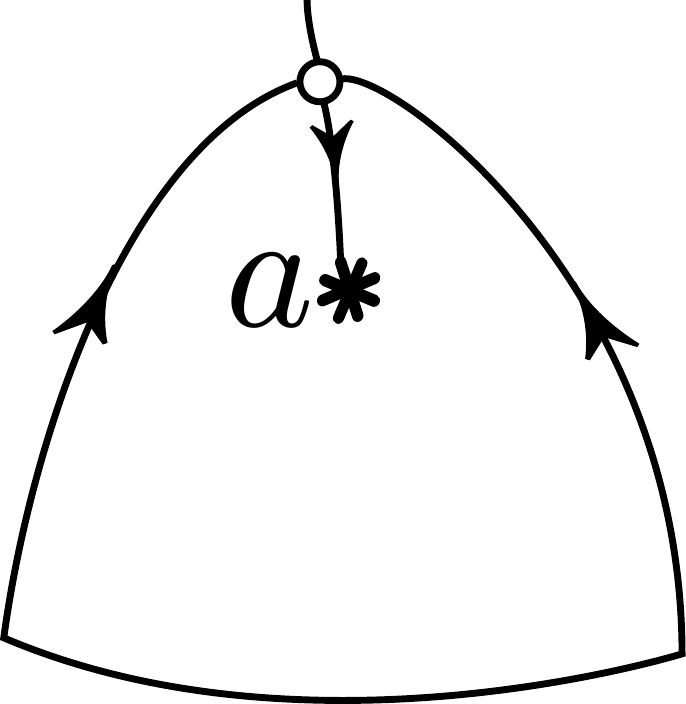}
\caption{Domains constructed in Step $b$}\label{fig:shapes}
\end{figure}

\begin{figure}
    \centering
\subcaptionbox{\label{fig:shapes1-a}}{\includegraphics[scale=0.5]{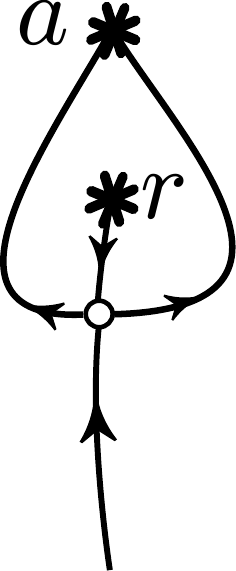}}
\hfil
\subcaptionbox{\label{fig:shapes1-b}}{\includegraphics[scale=0.5]{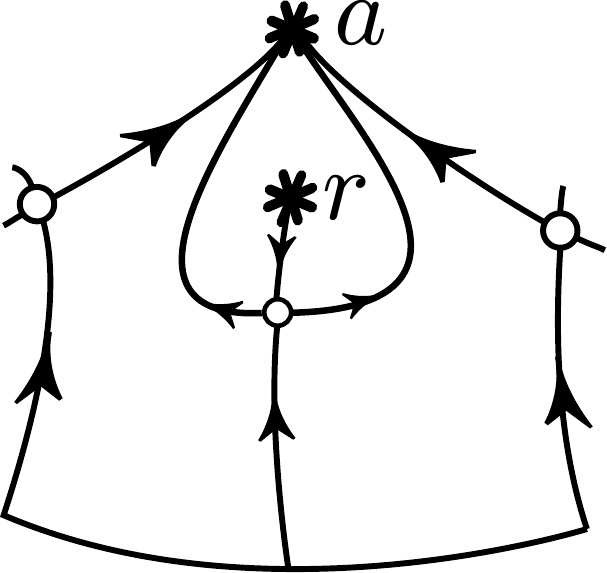}}
\hfil
\subcaptionbox{\label{fig:shapes1-c}}{\includegraphics[scale=0.5]{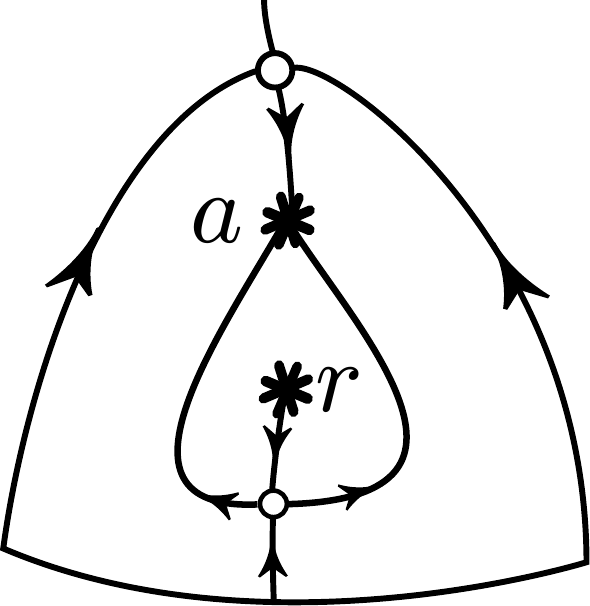}}
\caption{Constructing unstable separatrices in Step $c$}\label{fig:shapes1}
\end{figure}

On any arc of the curve $C$ which is an edge of the domains mentioned above, take all the points of the set $A$; each of them is a one-point equivalence class in the marked set $A$. Add to this point a graph shown at Figure~\ref{fig:shapes1-a}; the attractor $a$ on this figure should coincide with the attractor $a$ in the domain, see Figures~\ref{fig:shapes1-b}, \ref{fig:shapes1-c}. The construction of the separatrices of $v$ is over, see Figure~\ref{fig:realiz-c}. We can now construct the \vf $v$ with these separatrices  in $D$. This may be easily done in any connected component of the complement to  these separatrices in $D$.
It is easy to make the \vf $v$ smooth.
\end{proof}

\subsection{Realization Theorem for the sphere}

Here we prove  Theorem \ref{thm:real1}. The \vf will be constructed separately in the annular \nbd of the parabolic cycle, and in the components of its complement to the sphere. Begin with the annulus.

Let $U$ be the annulus $ 1 \le r \le 3$ in the polar coordinates $r, \a$. Consider a \vf $v_U$ in $U$ given by
$$ \dot \a = 1, \ \dot r = -(r-2)^2.$$
This system has  a parabolic cycle $\g$ given by $\{r=2\}$. Let $\G$ be a cross-section to $\g$ that belongs to the ray $\a = 0$. Let $C^{\pm}$ be the boundary circles of $U: C^- = \{r = 3\}, \ C^+ = \{r = 1\}$. Let $\ph^{\pm}=\ph_0^{\pm} $ be the coordinates on $C^{\pm}$ constructed for $v|_U$ as in Section \ref{subsec:non}. Let $A^{\pm} \subset C^{\pm}$ be the characteristic sets given in Theorem \ref{thm:real1}.

Let $D^{\pm}$ be the disc on the sphere disjoint from the interior of $U$ and bounded by $C^{\pm}: S^2 = D^+ \cup U \cup D^- $. By Lemma \ref{lem:marked}, \tes a smooth \vf $v^{\pm}$ on $D^{\pm}$ \st $A^{\pm}$ is a characteristic set of $v^{\pm}$ in $D^{\pm}$. Thus we have constructed a piecewise $C^{\infty}$ \vf $W$ on $S^2$ that is discontinuous on $C^{\pm}$. Let us change the smooth structure on $S^2$ in such a way that $W$ will become a $C^{\infty}$ smooth vector field. We will get a smooth \vf $W$ on a smooth manifold $M^2$ homeomorphic to a sphere. But \tes only one smooth structure on $S^2$. Hence, \tes a $C^{\infty}$ \diffeo $H: M^2 \to S^2$. The \vf $ v = H_*W $ is the desired one. This proves Theorem \ref{thm:real1}.

\section{Classification of $\pc$-families} \label{sec:class}

In this section we prove Theorem~\ref{thm:class}.


\subsection{Theorem \ref{thm:class} part 1}  \label{subsec:fampa}
Our goal is to prove that any pair of sing-equivalent local families gives rise to two pairs of characteristic sets that are equivalent in the sense of Definition~\ref{def:seteq:char}.
The heuristic proof is straightforward: as the families $V$ and $W$ are sing-equivalent, the homeomorphism $h$ identifies the values of $\e$ and $\d$ that correspond to saddle connections, thus identifies $\{\e_{kmn}\}$ and $\{\d_{kmn}\}$. Proposition \ref{prop-order-e} implies that the sets $\La (A^\pm )$ and $\La (B^\pm )$ are ordered in the same way. This implies the statement.

Let us pass to the detailed proof.

Let $V=\{v_{\e}\}$ and $W=\{w_{\delta}\}$ be two families satisfying assumptions of Theorem \ref{thm:class} part 1. Let $H=(h, H_{\e})$ be the sing-equivalence of $V,W$.
Following the previous sections, we denote by $\gamma$ the parabolic cycle of $v_0$, by $C^{\pm}$ its transversal loops, by  $l_j^{\pm}(\e)$ the separatrices of hyperbolic saddles $E_j(\e)$, $I_{j}(\e)$ of $v_{\e}$ that cross $C^{\pm}$. Let  $\tilde \gamma, \tilde C^{\pm}, \tilde l_j^{\pm}(\delta), \tilde E_j(\delta), \tilde I_{j}(\delta)$ be analogous objects for the family $W$ such that $\tilde l_j^{\pm} = H_0(l_j^{\pm})$, $\tilde E_j=H_0(E_j), \tilde I_{j}=H_0(I_j)$. We assume that the parabolic cycles disappear for $\e>0, \delta>0$, otherwise we reverse the parameter.

We will need the following lemma on sing-equivalence.
\begin{lem}
\label{lem:sing-sep}
In the above assumptions, $H_{\e}(l_i^{\pm}(\e)) = \tilde l^{\pm}_{i} (h(\e))$.
\end{lem}

This lemma implies Theorem \ref{thm:class} part 1.
Indeed, let $A^\pm$ be characteristic sets for $v_0$, and let $B^\pm $ be characteristic sets for $w_0$. Since the vector fields $v_0$ and $w_0$ are topologically conjugate, we have  $|A^+|=|B^+|$ and $|A^-|=|B^-|$. Let $\e_{kmn} $ be bifurcation parameters for the family $V$ and $\d_{kmn}$ be bifurcation parameters for $W$.

Now, for $\e=\e_{kmn}$, the separatrices $l_k^+(\e)$ and $l_m^-(\e)$ coincide, thus their images $H_{\eps}(l_k^+(\e)) = \tilde l_k^+(h(\e))$ and $H_{\eps}(l_m^-(\e)) = \tilde l_m^-(h(\e))$ coincide; here we use Lemma \ref{lem:sing-sep}. So the separatrices $\tilde l_{k}^+(h(\eps))$ and $\tilde l_m^-(h(\eps))$ of $w_{h(\eps)}$ form a saddle connection, and we conclude that $h(\e_{kmn}) = \delta_{kmn'}$ for any $n$ and some $n'$.

Finally, the homeomorphism $h$ takes the set $\{\e_{kmn}\}$ to the set $\{\delta_{kmn}\}$ and preserves $k,m$. However it may change $n$.

Note that $\{\e_{kmn} \mid n=n_0\}$ is the set of $|A^-|\cdot |A^+|$ subsequent numbers in the set $\{\e_{kmn}\}$, so $h$ takes them to $|A^-|\cdot |A^+| = |B^-|\cdot |B^+|$ subsequent numbers among $\{\delta_{kmn}\}$. Recall that changing $b^{\pm}$, we may achieve any cyclical shift of $n$ in the set $\{\e_{kmn}\}$, see Proposition \ref{prop-order-e-b}. So with a suitable choice of $b^{\pm}$, we may and will assume that $h(\e_{kmn}) = \delta_{kmn}$. By Proposition \ref{prop-order-e}, this implies that  the sets $\Lambda(A^{\pm})$ and $\Lambda(B^{\pm})$  are equivalent in the sense of Definition \ref{def:seteq:char}.

\begin{proof}[Proof of Lemma \ref{lem:sing-sep}]
We only prove the lemma for $l_i^-$; for $l_i^+$, the proof is analogous.
We have two topologically different cases: $l_i^-$ is the only separatrix of $E_i$ that intersects $C^-$, or both unstable separatrices of $E_i$ intersect $C^-$. We start with the second case.

Case 1. Consider a saddle $E$ whose two unstable separatrices wind towards $\gamma$. Let $L_1,L_3$ be these separatrices of $E$ (so $L_1=l_i^-$ and $L_3=l_j^-$ for some $i,j$; $E=E_i=E_j$), and let $L_2,L_4$ be two its stable separatrices. Let $R_2 = \alpha(L_2)$ and $R_4 = \alpha(L_4)$ be the $\alpha$-limit sets  of these separatrices; both of them are either hyperbolic singular points, or hyperbolic limit cycles. Note that  $R_2\neq R_4$ because these sets are on two different sides with respect to $\gamma\cup L_1\cup L_3$, see Fig. \ref{fig:wind-1}.

\begin{figure}
\subcaptionbox{\label{fig:wind-1}}{\includegraphics[width=0.35\textwidth]{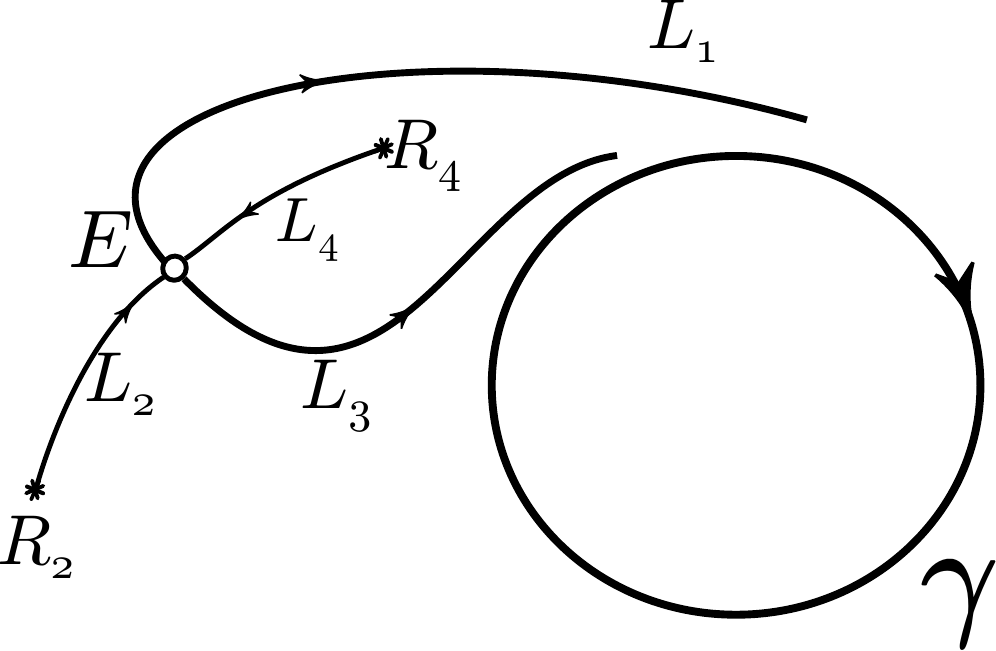}}
\hfil
\subcaptionbox{\label{fig:wind-2}}{\includegraphics[width=0.3\textwidth]{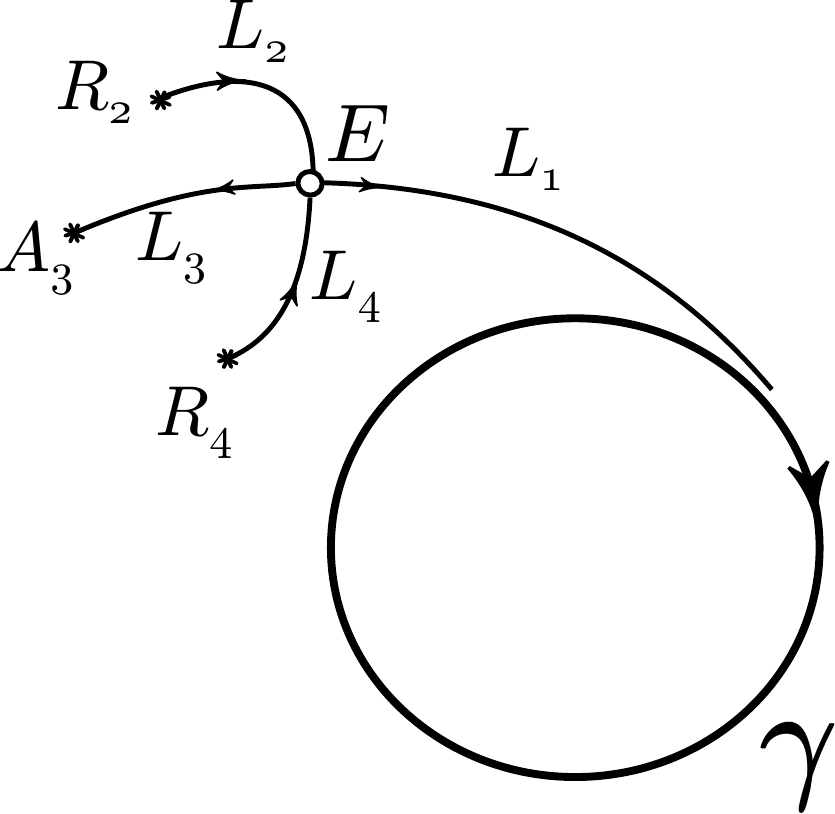}}
\caption{Two cases for Lemma \ref{lem:sing-sep}}\label{fig:wind-12}
\end{figure}

Let $E^\e, (L_i^\e,E^\e)$, and $R_i^\e$ be continuous families of singular points, germs of separatrices, and repellors of $v_{\e}$ such that $E^0=E, L_i^0=L_i, R_i^0=R_i$. Let $\tilde E, \tilde L_i, \tilde R_i$ be the images of $E, L_i$, and $R_i$ under $H_{0}$. Let $\tilde E^\d, (\tilde L_i^\d, \tilde E^\d), \tilde R_i^\d$ be continuous families of singular points, germs of separatrices, and repellors of $w_{\d}$.

The definition of sing-equivalence implies that $H_{\e}(E^\e) = \tilde E^{h(\e)}$ and $H_{\e}(R_i^\e) = \tilde R_i^{h(\e)}$. We should prove that the analogous statement holds for separatrices: $H_{\e}(L_i^\e)$ coincides with $\tilde L_i^{h(\e)}$. Clearly, $H_{\e}(L_1^{\e})$ is an unstable separatrix of $\tilde E^{h(\e)}$; two possibilities are $H_{\e}(L_1^\e) = \tilde L_1^{h(\e)}$ and $H_{\e}(L_1^\e) = \tilde L_3^{h(\e)}$. Our goal is to prove that the second case is impossible. Since $H_{\e}$ preserves orientation, it preserves a cyclical order of separatrices at $E^{\e}$, so in this case, we must have $H_{\e}(L_2^\e) = \tilde L_4^{h(\e)}$: in comparison with $H_0$, the map $H_{\e}$ rotates separatrices of $E^\e$.

Since $H_{\e}$ is a homeomorphism that conjugates $v_\e$ to $w_{h(\e)}$, it respects $\alpha$-, $\omega$- limit sets. So the $\alpha$-limit set of $H_{\e}(L_2^\e)$ must be $H_{\e}(R_2^{\e}) = \tilde R_2^{h(\e)}$. But the $\alpha$-limit set of $\tilde L_4^{h(\e)}$ is $\tilde R_4^{h(\e)}$. The contradiction shows that $H_{\e}(L_2^\e) \neq \tilde L_4^{h(\e)}$. Thus $H_{\e}(L_i^\e)=\tilde L_i^{h(\e)}$ for all $i$, which proves the lemma in Case 1.

Case 2. Suppose that  $E$ has only one unstable separatrix $L_1$ that winds towards $\gamma$, and other its separatrices $L_2,L_3,L_4$ have hyperbolic $\alpha$- and $\omega$-limit sets (see Fig. \ref{fig:wind-2}).  Then  similar arguments as in Case 1 apply to $L_3$. Namely, if the $\omega$-limit set of $L_3$ is $A_3$, then the separatrix $H_{\e}(L_3^{\e})$ must have the $\omega$-limit set $\tilde A_3^{h(\e)}$, while $\tilde L_1^{h(\e)}$ has its $\omega$-limit set inside $\gamma$. So $H_{\e}(L_3^\e)\neq \tilde L_1^{h(\e)}$, thus $H_{\e}(L_i^\e)=\tilde L_i^{h(\e)}$ for all $i$. This proves the lemma in Case 2.


\end{proof}

\subsection{Theorem \ref{thm:class} part 2} \label{subset:pafam}
Let $V=\{v_{\e}\}$ and $W=\{w_{\delta}\}$ be two families satisfying assumptions of Theorem \ref{thm:class}.
Let $\hat H$ be an orbital topological equivalence of $v_0$ and $w_0$.

Let $\gamma$ be a parabolic cycle of $v_0$; let $C^{+}$ and $C^-$ be its transversal loops. We assume that $C^-$ is outside $\gamma$ and $C^{+}$ is inside it.
Let $U$ be the open annulus bounded by $C^{-}$ and $C^{+}$. Let $D^{\pm}$ be the disc on the sphere disjoint from the interior of $U$ and bounded by $C^{\pm}$, so that $S^2=D^- \cup U \cup D^{+}$.

Let $\tilde \gamma, \tilde C^{\pm}, \tilde U, \tilde D^{\pm}$ be analogous sets for the family $W$. We may and will modify $\hat H$ so that $\hat H(C^\pm) = \tilde C^\pm$.
Recall that the choice of $\tilde C^{\pm}$ implies that the trajectories starting on  $\tilde C^-$ wind towards $\tilde \gamma$, and the trajectories starting on $\tilde C^+$ are repelled from $\tilde \gamma$. We assume that $\gamma, \tilde \gamma$ are oriented clockwise by time orientation.

We will use the notation of Section \ref{subsec: conn} for the family $V$: $l_j^{\pm}(\e)$ are separatrices of hyperbolic saddles $E_j(\e)$, $I_{j}(\e)$ of $v_{\e}$ that cross $C^{\pm}$. Let  $\tilde l_j^{\pm}(\delta), \tilde E_j(\delta), \tilde I_{j}(\delta)$ be analogous objects for the family $W$ such that $\tilde l_j^{\pm} = \hat H(l_j^{\pm})$, $\tilde E_j=\hat H(E_j), \tilde I_{j}=\hat H(I_j)$.

As in Section \ref{subsec: conn}, $a_{j}^{\pm}(\e)$ are intersection points of $l_{j}^{\pm}(\e)$ with $C^{\pm}$ in $\varphi^{\pm}_{\e}$-chart, and $\Delta_{\e}\colon C^- \to C^+$ is the Poincaré map along $v_{\eps}$. Let $\tilde a_{j}^{\pm}(\delta)$ and $\tilde \Delta_{\delta}$ be  the corresponding objects for $\{w_{\delta}\}$.

\subsection{Construction of the homeomorphism of bases $h\colon (\bbR,0) \to (\bbR,0)$}

 Assume that for families $V=\{v_{\e}\}$ and $W=\{w_{\delta}\}$, the parabolic cycle disappears for $\e>0$, $\delta>0$; otherwise we reverse the parameter.

 The equivalence of the characteristic sets for $V$ and $W$ (see Definition \ref{def:seteq:char}) implies that the numbers $\tau_{km}$ are ordered on $[0,1)$ in the same way as $\{\lambda_{km} +\alpha\}$ for some $\alpha$. Recall that $\lambda_{km}$ are well-defined modulo an additive constant that depends on the choice of coordinates on coordinate circles. Let us add a shift by $\alpha$ to the coordinate on $S^1_+$; this will add $\alpha$ to all numbers $\lambda_{km}$.
 Finally, we may and will assume that the numbers $\tau_{km}$ and $\lambda_{km}$ are ordered in the same way.

   Let $\{\e_{kmn}\}$ be the sequence of the bifurcation parameter values defined in Section \ref{subsec: conn} for the family $V$, and $\{\delta_{kmn}\}$ be the analogous sequence for the family $W$.
Proposition \ref{prop-order-e} implies that the order of numbers $\e_{kmn}$ and $\delta_{kmn}$, $n$ fixed, is the same. This implies that the sets $\{\e_{kmn}\}$ and $\{\delta_{kmn}\}$ may be identified by some homeomorphism $h\colon (\bbR,0) \to (\bbR, 0)$ with $h(\e_{kmn}) = \delta_{kmn}$.

In more detail, the homeomorphism $h$ is defined in the following way. We put $h(0)=0$ and $h|_{\eps<0} =id$, and for $\eps>0$, we take any homeomorphism that satisfies $h(\eps_{kmn}) = \delta_{kmn}$.
 We will need the following lemma.
 \begin{lem}
 \label{lem:order}
 In assumptions of Theorem \ref{thm:class} part 2, for a homeomorphism $h\colon (\bbR, 0) \to (\bbR, 0)$ constructed above and for any small positive $\e>0$, the points $a_{j}^{+}(\e)$ and $\Delta_{\e}(a_j^-(\e))$ are ordered along $C^+$ in the same way as the corresponding points $\tilde a_{j}^{+}(\delta)$ and $\tilde \Delta_{\delta}(\tilde a_j^-(\delta))$ on $\tilde C^+$ for $w_{\delta}$, where $\delta = h(\e)$.
 \end{lem}
  \begin{proof}
Note that if for small $\e$, two points $a_{m}^{+}(\e)$ and $\Delta_{\e}(a_k^-(\e))$ coincide, this implies $\e=\e_{kmn}$ for some $n$ (cf. Sec. \ref{subsec: conn}). Due to the construction of $h$, $h(\e) = \delta_{kmn}$, thus two corresponding points $\tilde a_{m}^{+}(\delta)$ and $\tilde \Delta_{\delta}(\tilde a_k^-(\delta))$ coincide.

We must also prove that for all $k_1,k_2,m$, the condition on $\e$
  \begin{equation}
  \label{eq:order-a}
   \Delta_{\e}(a_m^-(\e)) \in (a_{k_1}^{+}(\e), a_{k_2}^{+}(\e))
  \end{equation}
is equivalent to the condition on $\delta$
    \begin{equation}
  \label{eq:order-a'}
   \tilde \Delta_{\delta}(\tilde a_m^-(\delta)) \in (\tilde a_{k_1}^{+}(\delta), \tilde a_{k_2}^{+}(\delta))
  \end{equation}
  for small $\e$ and $\delta=h(\e)$ (on the right-hand side, we have oriented arcs of the oriented coordinate circles). This will finish the proof of the lemma.

Proposition \ref{prop:pmap} implies that \eqref{eq:order-a} is equivalent to  $(a_m^-(\e) -\tau(\eps) \mod 1) \in (a_{k_1}^{+}(\e), a_{k_2}^{+}(\e))$, i.e.
$$(-\tau(\eps) \mod 1)  \in (\tau_{k_1m}(\e), \tau_{k_2m}(\e))\subset \bbR/\bbZ.$$
Recall that $\tau(\eps)$ is monotonic for small $\eps$ with $\tau'(\eps) \to -\infty$ as $\e \to 0$, and the derivatives of $\tau_{km} (\e)$ are bounded. So the values of $\eps$ that satisfy the last inclusion are between the solutions $\e_{k_1mn}$, $\e_{k_2mn}$ of the connection equation \eqref{eqn:conn}.

In more detail, if the arc  $(\tau_{k_1m}(\e), \tau_{k_2m}(\e))$ of the circle does not contain zero (or, equivalently, the segment $(\tau_{k_1m}, \tau_{k_2m})$  does not contain zero), \eqref{eq:order-a} is equivalent to $ -\tau(\eps)+n \in (\tau_{k_1m}(\e), \tau_{k_2m}(\e))$ for some $n$, and the solutions are $\e \in (\e_{k_1mn}, \e_{k_2mn})$ for some $n$. If the  segment  $(\tau_{k_1m}, \tau_{k_2m})$  contains zero, \eqref{eq:order-a} is equivalent to $ -\tau(\eps)+n \in (\tau_{k_1m}(\e), \tau_{k_2m}(\e)+1)$ for some $n$, i.e. to $\e \in \bigcup_n (\e_{k_1mn}, \e_{k_2m({n-1})})$.

Finally, for small $\e$, the condition \eqref{eq:order-a} is equivalent to
\begin{align*}
&\e \in \bigcup_{n}(\e_{k_1mn}, \e_{k_2mn}) \text{ if  } 0 \in (\tau_{k_1m}, \tau_{k_2m})\\
&\e \in \bigcup_n (\e_{k_1mn}, \e_{k_2m({n-1})}) \text{ if  }0 \notin (\tau_{k_1m}, \tau_{k_2m})
\end{align*}

   Similarly,  the condition \eqref{eq:order-a'}
  is equivalent to
  \begin{align*}
&\delta \in \bigcup_{n}(\delta_{k_1mn}, \delta_{k_2mn}) \text{ if  } 0 \in (\lambda_{k_1m}, \lambda_{k_2m})\\
&\delta \in \bigcup_n (\delta_{k_1mn}, \delta_{k_2m({n-1})}) \text{ if  }0 \notin (\lambda_{k_1m}, \lambda_{k_2m})
\end{align*}

Since $\tau_{km}$ and $\lambda_{km}$ are ordered in the same way, the construction of $h$ above implies that the conditions  \eqref{eq:order-a} and \eqref{eq:order-a'} are equivalent for $\delta=h(\e)$.

  \end{proof}
\subsection{Construction of equivalence of $V,W$ in the neighborhoods of parabolic cycles}
 In this section, we construct a required homeomorphism conjugating $v_{\eps}$ to $w_{h(\eps)}$ in the neighborhoods of the parabolic cycles $\gamma$, $\tilde \gamma$. We will extend it to the whole sphere in the next section.

 \begin{lem}
 \label{lem-H-mid}
  In assumptions of Theorem \ref{thm:class} part 2, for $h$ constructed above and for each small $\e$,
   there exists a homeomorphism $H_\e^{mid} \colon U\to \tilde U$   such that
   \begin{equation}
   \label{eq:Hmid}
    H_\e^{mid} (U \cap l_j^{-}(\e)) = \tilde U \cap \tilde l_j^{-}(\delta) \quad \text{ and } \quad H_\e^{mid}(U \cap l_j^{+}(\e)) = \tilde U \cap \tilde l_j^{+}(\delta),
   \end{equation}
 where $\delta = h(\e)$.
 \end{lem}
\begin{proof}
 For $\e>0$, the parabolic cycle disappears, and both vector fields $v_{\eps}|_{U}$, $w_{\delta}|_{\tilde U}$ are conjugate to the radial vector field in the annulus $1<r<2$. So the statement follows from \autoref{lem:order}.

 For $\e=0$, we may take $H_\e^{mid}=\hat H|_{U}$.

 For $\e<0$, the parabolic cycle splits into two. Both vector fields $v_{\e}|_{U}$ and $w_{\delta}|_{\tilde U}$ have two hyperbolic cycles in $U, \tilde U$, the outer one is attracting and the inner one is repelling; both oriented clockwise by time orientation. We conclude that $v_{\e}|_{U}$ and $w_{\delta}|_{\tilde U}$ are conjugate. It is easy to modify the conjugacy so that it takes characteristic sets on $C^{\pm}$ to characteristic sets on $\tilde C^{\pm}$, i.e.  the conditions \eqref{eq:Hmid} are satisfied.
\end{proof}

\subsection{Proof of Theorem \ref{thm:class} part 2  modulo \autoref{lem:flex} on flexibility of authomorphisms}

Recall that $v_{0}$ is topologically conjugate to $w_0$, and the conjugacy $\hat H$ takes $l^{\pm}_j$ to $\tilde l^{\pm}_j$.

Now, $v_{0}|_{D^\pm}$ is structurally stable, so $v_{\e}|_{D^\pm}$ is conjugate to it. Similarly, $w_{0}|_{\tilde D^{\pm}}$ is conjugate to $w_{h(\e)}|_{\tilde D^\pm}$. Thus there exist two homeomorphisms $H_{\e}^+ \colon D^+ \to \tilde D^+$ and $H_{\e}^-\colon D^- \to \tilde D^-$ that conjugate $v_{\e}|_{D^{\pm}}$ to $w_{\delta}|_{\tilde D^{\pm}}$ and take $l^{\pm}_j (\e)$ to $\tilde l^{\pm}_j (\delta)$, where $\delta=h(\e)$.

Now let us construct $H_{\e}$.

For $\e=0$, we simply take $H_{\e}=\hat H$.
Our goal for $\e \neq 0$ will be to agree  $H_{\e}^{\pm}$  with $H_{\e}^{mid}$ constructed in \autoref{lem-H-mid}.

We will use the following lemma. Let $\Sep v$ be the union of all separatrices of $v$, $\Per v$ be the union of all limit cycles of $v$, $\Sing v$ be the union of all singular points of $v$.

\begin{lem}[On flexibility of automorphisms]
\label{lem:flex}
   Let $v$ be a smooth Morse-Smale vector field in a closed disc $D$ with smooth boundary. Let $v$ be transversal to $\partial D$.

   Then any orientation-preserving homeomorphism $g\colon \partial D \to \partial D$ such that  $g$ is identical on $\partial D \cap \Sep v$ extends to an orbital topological automorphism $G\colon D\to D$ of the vector field $v$, and $G|_{\Sing v \cup \Per v} =id$.
\end{lem}
The proof of this lemma constitutes Section \ref{sec:flex} below. Here we finish the proof of Theorem \ref{thm:class} modulo Lemma \ref{lem:flex}.

In order to agree $H^{+}_{\e}$ to $H^{mid}_{\e}$ on $C^+$, we apply this lemma to the vector field $v_{\e}|_{D^{+}}$ and the homeomorphism $g = (H_{\e}^{+})^{-1} \circ H_\e^{mid}$  on $C^{+}$, and get an automorphism $G^+$ of $v_{\eps}|_{D^+}$. Now $H_{\e}^{+} \circ G^+$ is a sing-equivalence of $v_{\e}$ and $w_{\delta}$ that takes $D^{+}$ to $\tilde D^+$,  coincides with $H_\e^{mid}$ on $C^{+}$ and conjugates $v_{\e}$ to $w_{h(\e)}$ in $D^{+}$ and $\tilde D^{+}$. Similarly, we construct a homeomorphism $H_{\e}^- \circ G^-$ in $D^-$ that coincides with $H_\e^{mid}$ on $C^-$ and conjugates $v_{\e}$ to $w_{h(\e)}$. These two homeomorphisms $H_{\e}^\pm \circ G^\pm$ and $H^{mid}_\e$ glue into the sing-equivalence $H_\eps$ on the whole sphere, and the proof of Theorem \ref{thm:class} is complete.

\subsection{Flexibility of automorphisms}
\label{sec:flex}
This section is devoted to the proof of \autoref{lem:flex}.
We start with some general definitions.

\begin{defin}
Let $v$ be a Morse-Smale vector field.
The union of all singular points, separatrices and limit cycles of $v$ is called a \emph{separatrix skeleton} of $v$. The connected components of the complement to the separatrix skeleton are called \emph{canonical regions}.
\end{defin}

Clearly, each canonical region $R$ has a common $\alpha$- and $\omega$-limit set; we denote them by $\alpha(R)$ and $\omega(R)$.
The following statement is formulated in \cite{DLA}.
\begin{prop}
\label{prop:canonreg}
 Any canonical region $R$ of a $C^2$-smooth vector field on $S^2$ is parallel, i.e. $v|_{R}$ is topologically equivalent to one of the following:
\begin{itemize}
 \item strip flow: $\partial / \partial x$ in the strip $\Pi := \bbR\times (0,1)$;
 \item spiral flow: $\partial /\partial r$ in $\bbR^2\setminus 0$, where $(r,\phi)$ are polar coordinates;
 \item annular flow: $\partial /\partial \phi$ in $\bbR^2\setminus 0$.
\end{itemize}
\end{prop}

The last case does not appear for Morse-Smale vector fields. The possible shapes for canonical regions of Morse-Smale fields are shown on Fig. \ref{fig:canonReg}, see also \cite[Sec. 1.2]{N}. In particular, each strip canonical region is bounded by four or three separatrices.

\begin{figure}
\centering
\subcaptionbox{\label{fig:canonReg-a}}{\includegraphics[scale=0.75]{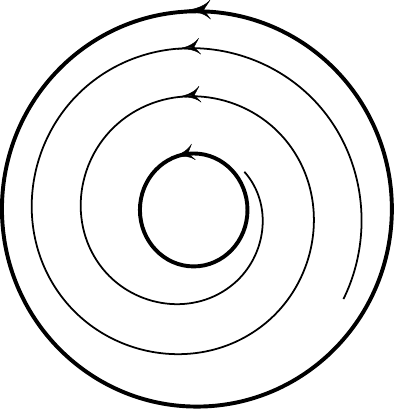}}
\hfil
\subcaptionbox{}{\includegraphics[scale=0.75]{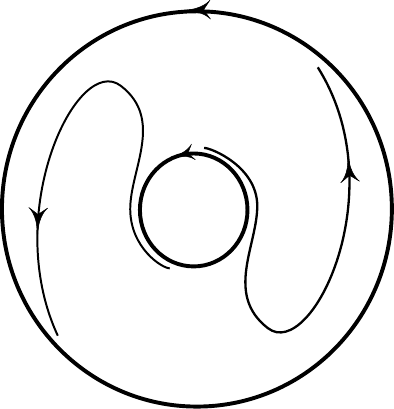}}
\hfil
\subcaptionbox{}{\includegraphics[scale=0.75]{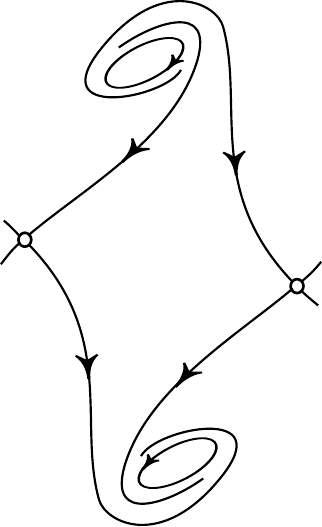}}
\hfil
\subcaptionbox{}{\includegraphics[scale=0.75]{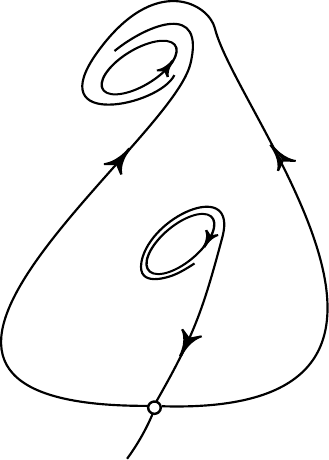}}
 \caption{Possible shapes of strip and spiral canonical regions for Morse-Smale vector fields. Each limit cycle may be replaced by a singular point}\label{fig:canonReg}
\end{figure}

Proposition \ref{prop:canonreg} provides us by continuous charts in $R$ that conjugate $v$ to standard vector fields. This will enable us to prove the following propositions.
\begin{prop}
\label{prop:extension}
 Let $R$ be a \textbf{strip} canonical region of a Morse-Smale vector field $v$ on the sphere. Let $T$ be a transversal to $v$ that intersects all trajectories of $v|_{R}$. Let endpoints of $T$ be located on separatrices that bound $R$.

 Then given a homeomorphism $g\colon T \to T$ of a transversal that fixes endpoints of $T$, we may extend it to an automorphism $G\colon \overline R \to \overline R$ of $v$ such that $G|_{\partial R}=id$.
\end{prop}

\begin{prop}
\label{prop:extension-2}
  Let $R$ be a \textbf{spiral} canonical region of a Morse-Smale vector field $v$ on the sphere. Let $T$ be a closed transversal to $v$ that intersects all trajectories of $v|_{R}$.

 Then given an orientation-preserving homeomorphism $g\colon T \to T$ of a transversal, we may extend it to an automorphism $G\colon \overline R \to \overline R$ of $v$ such that $G|_{\partial R}=id$.
\end{prop}

The idea of their proofs is to construct the extension $G$ in the canonical charts provided by Proposition \ref{prop:canonreg}. We may take $G$ that preserves the trajectories of $v$ in the canonical chart and coincides with $g$ on $T$. The continuity of $G$ on $\partial R$ requires a little more caution.

Now we prove Lemma \ref{lem:flex} modulo these propositions.

\begin{proof}[The proof of \autoref{lem:flex}]
Without loss of generality, assume that on $\partial D$, the vector field $v$ points inside $D$.
For $v|_{D}$, we are going to refer to the general statements that hold for vector fields on the sphere. So we will extend $v$ smoothly to the complement of $D$ by a radial vector field having one hyperbolic source in $S^2\setminus D$. The vector field $\hat v$ thus obtained is Morse-Smale.

Define the required homeomorphism $G$ to be identical on the separatrix skeleton of $\hat v$.
If a canonical region $R$ of $\hat v$ does not intersect $\partial D$, define $G|_{R}:=id$.
If  $\partial D$ intersects several strip canonical regions, the map $G$ in all  such regions is provided by Proposition \ref{prop:extension} above. Finally, if there exists a spiral canonical region $R$ that contains the whole  $\partial D$, the map $G$ is provided by Proposition \ref{prop:extension-2}.
\end{proof}
We conclude by proving Propositions \ref{prop:extension} and \ref{prop:extension-2}.
\begin{proof}[Proof of Proposition \ref{prop:extension}.]
 We will need a continuous chart $\psi \colon \Pi \to R$ that conjugates $\partial/\partial x$ to $v$ (as the one provided by Proposition \ref{prop:canonreg}), but with the following additional properties.
 \begin{itemize}
 \item $\psi$ takes the vertical open segment $\tilde T=\{0\}\times (0,1) \subset \Pi $ to $T$.
  \item $\psi$ extends continuously to $\partial \Pi$. $\psi^{-1}$ extends continuously to $\psi(\bbR \times \{0\})$ and $\psi(\bbR \times \{1\})$ (however it is possible that $\psi^{-1}$ is not well-defined on $\partial R$, see Fig. \ref{fig:canonReg}d);
  \item The diameter of a transversal $\psi (\{x\} \times (0,1))$ to $v$ tends to zero as $x\to \pm \infty$.
 \end{itemize}
To satisfy all properties, we act in the following way. Let $U(\alpha(R))$ and $U(\omega(R))$ be small neighborhoods of $\alpha(R), \omega(R)$. We take the chart $\psi$ provided by Proposition \ref{prop:canonreg}, then shift and rescale it on each horizontal line in $\Pi$ so that it takes $\tilde T$ to $T$ and provides a natural parametrization on each trajectory of $v$ in  $R\setminus U(\alpha(R)) \setminus U(\omega(R))$. This ensures the first two requirements in $R\setminus U(\alpha(R)) \setminus U(\omega(R))$. The set $\overline R \cap U(\alpha(R))$ is a family of non-singular trajectories of $v$, and it is easy to choose $\psi^{-1}$ on this set so that the second requirement is still satisfied; the same holds for $R \cap U(\omega(R))$.   After all,  $\psi$ satisfies the first and the second requirement.

Now if both $\alpha(R)$ and $\omega(R)$ are singular points, the third requirement is automatically satisfied: $\psi(\{x\} \times (0,1))$ is in a small neighborhood of $\alpha(R)$ or $\omega(R)$ for $x$ close to $\pm\infty$, so has a small diameter. If $\alpha(R)$ is a limit cycle $c$, we further modify $\psi$ in $U(\alpha(R))$. Namely, we choose a continuous family of transversals to $v$: one transversal at each point of $c$. Then we modify $\psi^{-1}$ near $c$ so that it takes the intersections of these transversals with $R$ to vertical segments in $\Pi$. We do the same in a small neighborhood of $\omega(R)$ if this set is a limit cycle as well. Clearly, all the three requirements are satisfied after such modifications.

Now, let $\hat  g:=\psi^{-1}g\psi$ be the map $g$ in $\psi$-chart. Then it extends to the map $\hat G \colon \Pi\to \Pi$ given by $\hat G(x,y) = (x, \hat g(y))$; note that $\hat G$ preserves vertical segments in $\Pi$. Since $g$ fixes endpoints of $T$, $\hat g$ fixes endpoints of $\hat T$. So $\hat G$ is identical on the upper and the lower border of $\Pi$.

Let $G\colon R\to R$ be the map $\hat G$ in $\psi^{-1}$-chart, i.e. $G=\psi \hat G \psi^{-1}$; clearly, $G|_T=g$. Now, $G$ is identical on the separatrices that bound $R$, due to the continuity of $\psi, \psi^{-1}$ and the corresponding property of $\hat G$. Moreover, if we put $G|_{\alpha(R)}=id, G|_{\omega(R)}=id$, the map $G$ is still continuous. Indeed, since $\hat G$ preserves vertical segments in $\Pi$, the map $G$ near $\alpha(R), \omega(R)$ takes transversals $\psi(\{x\}\times (0,1))$ into themselves, and the diameter of these transversals tends to $0$ (see the third requirement on $\psi$); this implies the statement.

Finally, $G|_{\partial R}=id$. This finishes the proof.
\end{proof}

\begin{proof}[Proof of Proposition \ref{prop:extension-2}.]
The proof is analogous to the proof of Proposition \ref{prop:extension} but simpler, because we only need the continuity of $G$ on $\alpha(R), \omega(R)$. Note that the vector field $d/dr$ in  $\bbR^2\setminus \{0\}$  is topologically equivalent to the vector field $d/dx$ in the strip with identified borders $\Pi^* := \bbR^2/ ((x,y) \sim (x,y+2 \pi) )$. We will find a continuous chart $\psi \colon \Pi^*  \to R$ that conjugates $\partial/\partial x$ to $v$ and has two additional properties:
 \begin{itemize}
 \item $\psi$ takes the vertical segment $\tilde T=\{0\}\times [0,1] \subset \Pi^* $ to $T$.
  \item The diameter of a transversal $\psi (\{ y=x+ n\} )$ to $v$ tends to zero as $n \to \pm \infty$.
 \end{itemize}
As before, if $\alpha(R)$ and $\omega(R)$ are singular points, the second requirement is trivial. To satisfy this requirement near a limit cycle, we choose a family of small transversals to this cycle and require that $\psi$ takes slanted segments $\{y=x+n\}$, $n$ large, to these transversals. The definition of $\hat g$ is as before. When we extend $\hat g$ to $\hat G$, we choose $\hat G$ so that it preserves horizontal lines and also preserves slanted segments $\{ y=x+ n\}$ for large $n$. As before, $G=\psi \hat G \psi^{-1}$ is a continuous map in $R$ that satisfies $G|_{T}=g$. Moreover, we may extend $G$ continuously to $\alpha(R), \omega(R)$ by the identity map, because the map $G$ preserves short transversals $\psi (\{ y=x+ n\} )$ to $v$. This finishes the proof.
\end{proof}

\section{Structural stability}
\label{sec-str}
In this section Theorem~\ref{thm:stab1}, hence, Theorem~\ref{thm:stab}, is proved.

\subsection{Reduction to the Classification theorem.}

\begin{prop}   \label{prop:close}   For $C^4$-close \vfs $v, w$ of class PC, the corresponding characteristic pairs of sets $A^\pm (v), A^\pm (w)$ are respectively close.
\end{prop}

This proposition is proved in the next two sections. Let us now check that two close local $\pc$ families  satisfy the assumptions of Theorem \ref{thm:class} part 2.  This will imply Theorem~\ref{thm:stab1}.

Let $V$ be an unfolding of $v$: $V = \{ v_\e |\e \in (\rr ,0)\} , \ v_0 = v$. Let $w \in PC $ be so close to $v$ that the Sotomayor theorem is applicable: $w$ is orbitally topologically equivalent to $v$. Moreover, let $v$ and $w$ be so close that the corresponding pairs of characteristic sets $A^\pm (v), A^\pm (w)$ are close, see Proposition \ref{prop:close}. Then they are equivalent in the sense of Definition~\ref{def:seteq:char}.  Now let $W$ be a $\pc$-family that unfolds $w$. All the assumptions of Theorem~\ref{thm:class} part 2 for the families $V$ and $W$ are justified. Hence, they are equivalent. Therefore, the family $V$ is structurally stable. Theorem~\ref{thm:stab1} is proved modulo Proposition~\ref{prop:close}.

\subsection{Takens theorem with a parameter}

The main step in the proof of Proposition~\ref{prop:close} is to check that canonical coordinates on the \nccs for close \vfs of class $PC$ are also close.
Equivalently, we should prove that the time functions $T^{\pm}$ defined in Sec. \ref{sub:marfin} are close for close vector fields. It is sufficient to check that the generators of close parabolic germs are close.

\begin{prop}   \label{prop:takpar}    Suppose that two parabolic germs are $C^4$-close. Then their generators are $C$-close. In more detail, for any
parabolic germ \tes a representative $P$ with the following property. Let $U$ be the domain of $P$. Then \tes a \nbd $V \subset U$ of 0 \st any map $Q$
which is sufficiently $C^4$-close to $P$ in $U$ has a generator $C$-close to that of $P$ in $V$.
\end{prop}

\begin{proof}   First, recall the main steps of the proof of Takens Theorem (see Theorem \ref{thm:pmap}), according to \cite{M}  and \cite{I90}.  Let $P (x)= x + x^2 + (a + 1)x^3 +\dots $ be a real smooth germ. Then $P$ is formally equivalent to the time one shift $P_0$ along the vector field $u_a$:
\be  \label{eqn:norm}
P_0 = g^1_{u_a}, \ u_a = \frac {x^2}{1-ax}.
\ee
The maps $P$ and $P_0$ have the same $3$-gets at $0$. Hence,
\begin{equation}\label{eqn:r}
P = P_0 + R, \ |R| \le C|x^4|\text{ and } |R'|\le C' |x^3|
\end{equation}
for some $C, C' > 0$. In what follows, different constants depending on the functions considered are denoted by $C$ with subscripts and superscripts. The chart
$$
t = -\frac 1 x - a\ln x
$$
rectifies the \vf $u_a$ and brings a \nbd of $0$ to a \nbd of infinity. The maps $P$ and $P_0$ written in the chart $t$ are denoted by $\hat P, \hat P_0$. Clearly, $\hat P_0$ is a mere translation by $1$: $\hat P_0(t) = t + 1$. Let
$\hat P = \hat P_0 + \hat R = t + 1 + \hat R$. Then $|\hat R| < C|t^{-2}|$, see Proposition \ref{prop:Tak-1} below. Let us find a map $H = id + h$ defined near infinity that conjugates $\hat P$ and $\hat P_0$. This map is found separately in two half-\nbds $(\rr^-, \infty ), \ (\rr^+, \infty )$. Let us find it in $(\rr^+, \infty )$. Note that $h$ satisfies $ \hat P_0 \circ (id +h)=(id+h)\circ \hat P $, which implies the Abel equation on $h$:
$$
h = h \circ \hat P + \hat R.
$$
The solution of this equation in $(\rr^+, \infty )$ has the form:
\begin{equation}\label{eqn:abel}
h^+ = \sum_{k=0}^\infty \hat R \circ \hat P^k.
\end{equation}
Due to Proposition \ref{prop:Tak-2} below, this series converges on $\rr^+$ near infinity; moreover, if $C$, $C'$ in \eqref{eqn:r} are small, then $h$ is $C^1$-small.  Finally, we have found the desired generator $u$ of $P$.  In the coordinate $H=id+h^+$, this generator is a unit vector field $\mathbf{e}$. In the coordinate $t$, it equals $(H^{-1})_* \mathbf{e}$. In the initial coordinate, it equals $u = (t^{-1} \circ H^{-1})_* \mathbf{e}$.
Together with~\eqref{eqn:abel}, this provides the desired formula for $u$ in $(\rr^+, 0)$.

A similar formula holds in $(\rr^-,0)$ with the only difference that in this case, the solution of the Abel equation is
\begin{equation}\label{eqn:abel1}
h = h^- = -\sum_{k=1}^\infty \hat R \circ \hat P^{-k}.
\end{equation}

We stop here and do not check the assertion of Theorem \ref{thm:pmap} that $u$ is infinitely smooth. See \cite{M}  and \cite{I90} for the rest of the proof of Takens theorem.

\begin{prop}
\label{prop:Tak-1}
In the above assumptions, $|\hat R| < C_1|t^{-2}|$ and $|\hat R'|\le C_1' |t^{-3}|$; if the constants $C,C'$ in \eqref{eqn:r} are small, then $C_1, C_1'$ are also small.
\end{prop}

\begin{proof} By definition of $\hat R$, $\hat R = \hat P-\hat P_0$. So
$$
\hat R \circ t = t \circ P - t \circ P_0 = t\circ (P_0+R) - t\circ P_0.
$$
Hence, for some $\theta \in [0,1]$,
$$
|\hat R(t(x))| \le  t'\circ (P_0 + \theta R) \cdot|R(x)| \le \tilde C x^{-2}\cdot x^4 = \tilde C x^2 < C_1 t^{-2}.
$$
Moreover, by the same argument, for some $\theta\in [0,1]$,
$$
\left|\frac {d}{dx}\hat R \circ t \right|  = \left|t'\circ (P_0+R) \cdot (P_0'+R') - t'\circ (P_0) \cdot P_0' \right| \le \left|t''\circ (P_0 + \theta R)\cdot R \cdot P_0'\right| + \left|t'\circ (P_0 + R)\cdot R'\right| \le \tilde C_1 |x|;
$$
dependence on $x$ in the left and the middle part of the display is skipped for brevity. Hence,
\begin{equation}\label{eqn:deriv}
|\frac {d}{dt}\hat R| \le \frac{|\frac {d}{dx}\hat R \circ t|}{|t'(x)|} \le \tilde C'_1 {|x|}^3 \le C_1'{|t|}^{-3}.
\end{equation}
Note that if $C$, $C'$ in \eqref{eqn:r} are small, then $C_1$ and $C_1'$ are also small.
\end{proof}

\begin{prop}
 \label{prop:Tak-2}
In the above assumptions, the series for $h^+$ converges in $(\bbR, +\infty)$; if the constants $C$, $C'$ in \eqref{eqn:r} are small, then $h^+$ is small in $C^1$ metric.
\end{prop}

\begin{proof} The series~\eqref{eqn:abel} converges because $\hat R$ decreases as $C_1 t^{-2}$, and $\hat P^k$ increases as an arithmetic progression. Moreover, for small  $C$, $C'$ in \eqref{eqn:r}, $C_1$ is also small, thus $h$ is small in $C$ metric near infinity.

To estimate $h'$, note that $h' = \sum (\hat R' \circ \hat P^k) \cdot (\hat P^k)'.$ The function $\hat R'$ is small by \eqref{eqn:deriv}.
The sequence $\hat P^k$ increases as an arithmetic progression. It remains to prove that the sequence $(\hat P^k)'$ is bounded; this will imply that $h'$ is small in $C$-metric.

Let us estimate $\hat P^k(t)'$  for $t \in \rr^+$ large. This derivative is a product of values of the function $\hat P = 1+ \hat R'$ along the first $k$ points of the orbit of $t$ under the map $\hat P$. This orbit grows as an arithmetic progression. The logarithm of the product mentioned above is no greater than the sum $ c \sum_{k=1}^{\infty} |\hat R' \circ \hat P^k (t)|$. By \eqref{eqn:deriv}, this sum is uniformly bounded in a \nbd of infinity. Hence, the sequence $(\hat P^k)'$ is uniformly bounded for large $t$ as required.
\end{proof}

Let us now prove Proposition~\ref{prop:takpar}. Takens theorem implies that the normalizing chart that conjugates $P$ and $P_0$ is infinitely smooth. From now on, we switch to this chart; this reduces the general case to the case when $P = g_{u_{a}}^1$ for some $a$, and $Q$ is $C^4$-close to it.

Put $Q=Q_0+R$ where $Q_0 = g_{u_b}^1$ for some $b$. Since $a = \frac {P^{(3)}(0)}{6} - 1, \ b = \frac {Q^{(3)}(0)}{6} - 1,$ we conclude that $a$ and $b$ are close. As above, $|R| \le C x^4$, $|R'|\le C' |x^3|$; moreover, $C, C'$ are small, because $P-P_0$ is zero and $Q,Q_0$ are $C^4$-close to $P,P_0$.

Now, let us repeat the arguments from the proof of Takens theorem for $Q = Q_0+R$.  Let $t_b = -\frac 1x + b \ln x$ be the rectifying chart for the \vf $u_b$, and $\hat Q = t + 1 + \hat R$ be the map $Q$ written in this chart. Let  $h$ be the map given by~\eqref{eqn:abel} and~\eqref{eqn:abel1} for this $\hat R$; put $H=id+h$. Then the generator $u_Q$ for the map $Q$ is given by $u_{Q} = (t_b^{-1} H^{-1})_* \mathbf e$. Due to Proposition \ref{prop:Tak-2}, $h$ is $C^1$-small, thus $H$ is $C^1$-close to identity; so $u_{Q}$ is close to $(t_b^{-1})_{*} \mathbf e = u_b$. Finally,  $u_b$ is close to $u_a$ because $a$ and $b$ are close as we showed above. So $u_Q$ is close to $u_a$. This proves Proposition~\ref{prop:takpar}.

\end{proof}

\subsection{Proximity of the characteristic sets}

Here we complete the proof of Proposition~\ref{prop:close}.

Let $v$ and $w$ be two close \vfs of class $PC, \g $ and $\t \g $ be their (close) parabolic cycles, and $C^\pm $ be their common \nccs that separate the parabolic cycles from the rest of the sphere. The \vfs $v$ and $w$ are orbitally topologically equivalent as explained above. Let $E_m, \t E_m, \ m = 1, \dots , M$ and $I_k, \t I_k$, $k = 1,..., K$ be the saddles of $v$ and $w$ respectively, whose separatrices wind to $\g $ and $\t \g $ as described above; $E_m$ and $\t E_m$, $I_k$ and $\t I_k$ are close to each other.

Then the intersection points of the separatrices of these saddles with the \nccs $C^{\pm}$ are close.   But we have to prove that these points are close on the coordinate circles with the canonical coordinates $\ph_0^{\pm} , \t \ph_0^{\pm}$. The latter statement follows from Proposition~\ref{prop:takpar}. This proves Proposition~\ref{prop:close}, and completes the proof of the Structural stability Theorems~\ref{thm:stab} and~\ref{thm:stab1}.

\section{Bifurcation support vs \lbs }
\label{sec-support}
In \cite{AAIS},    Arnold introduced a notion of a \emph{bifurcation support}. 
Begin with the quotation from Arnold.

\emph{Although even \lb in high codimensions  (at least three) on a disc are not fully investigated, it is natural to discuss \nb in multiparameter families
of \vfs on a two-dimensional sphere. For their description, it is necessary to single out the set of trajectories defining perestroikas in these families.}

\begin{defin}
  A finite subset of the phase space is said to \emph{support a bifurcation} if there exists an arbitrarily small \nbd of this subset and
a \nbd of the bifurcation values of the parameter (depending on it) such that, outside this \nbd of the subset, the deformation
(at values of the parameter from the second \nbd )   is topologically trivial.
\end{defin}

{\begin{defin}    The \emph{\bif support} of a \bif is the union of all minimal sets supporting a \bif
(``minimal'' means not containing a proper subset that supports a bifurcation).
\end{defin}

\begin{defin}    Two deformations of \vfs with \bif supports $\Sigma_1$ and $\Sigma_2$ are said to be \emph{equivalent on
their supports}  if there exist arbitrarily small \nbds of the supports, and \nbds of the \bif values of the parameters
depending on them, such that the restrictions of the families to these \nbds of the supports are topologically equivalent,
or weakly equivalent, over these \nbds of \bif values.
\end{defin}

The quotation ends here. The following theorem shows that the bifurcation support is insufficient for the description of the bifurcations.

\begin{thm}   \label{thm:lbs}    There exist two orbitally topologically equivalent \vfs of class $PC$, whose generic unfoldings in one-parameter families are equivalent on
their supports, but not sing-equivalent on the whole sphere.
\end{thm}

This is an improved version of Theorem \ref{thm:noneq}.

\begin{proof}    Consider a \vf $v$ of class $PC$. A bifurcation carrier is an arbitrary point on the parabolic cycle $\g $ of this field. The bifurcation support is the cycle $\g $ itself. Under the unfolding of $v$, the cycle $\g $ splits in two on one
side of the critical value of the parameter, and vanishes on the other side. For any two \vfs of class $PC$ their deformations are equivalent on
their supports.

Consider now two \vfs of class $PC$ with non-equivalent pairs of characteristic sets, see Figure \ref{fig:noneq} for example. By Theorem \ref{thm:class}, the unfoldings of these fields are not sing-equivalent. \end{proof}

\begin{figure}
   \centering
\includegraphics[scale=0.3]{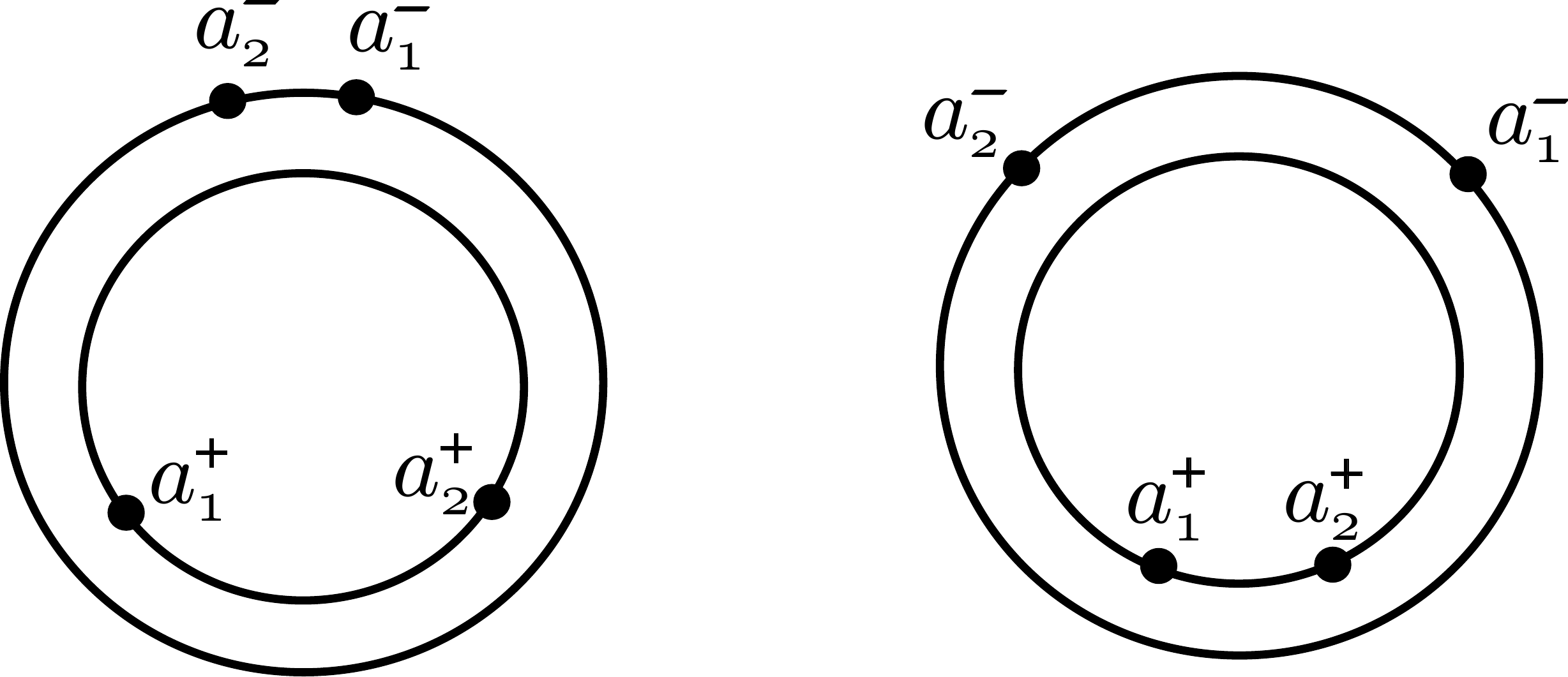}
\caption{Non-equivalent pairs of characteristic sets on coordinate circles}\label{fig:noneq}
\end{figure}
This proves Theorem~\ref{thm:lbs}. Simultaneously Theorem~\ref{thm:noneq} is proved.

\end{document}